\documentclass[10pt]{amsart}
\usepackage{amsthm, amsmath,cleveref, amsfonts, amssymb, enumerate, textcmds, tensor, enumitem, mathdots}
\usepackage{pst-node}
\usepackage{tikz-cd} 
\usepackage{graphicx,tikz}
\usepackage{enumerate}
\usepackage{pinlabel}
\usepackage[colorinlistoftodos]{todonotes}
\newcommand{\R}{\mathbb R}

\newcommand{\N}{\mathbb N}

\newcommand{\p}{\mathfrak{p}}
\newcommand{\g}{\mathfrak{g}}
\newcommand{\kk}{\mathfrak{k}}

\newcommand{\ft}{\mathfrak{t}}

\DeclareMathOperator{\dbar}{\overline{\partial}}

\newtheorem{thm}{Theorem}[section]
\newtheorem{lem}[thm]{Lemma}

\newtheorem{prop}[thm]{Proposition}
\newtheorem{cor}[thm]{Corollary}

\newtheorem*{thma}{Theorem A}
\newtheorem*{thmb}{Theorem B}
\newtheorem*{thmc}{Theorem C}
\newtheorem*{thmd}{Theorem D}

\newtheorem*{cora}{Corollary A}

\newtheorem*{corc}{Corollary C}

\newtheorem*{core}{Corollary E}
\newtheorem*{thme}{Theorem E}

\theoremstyle{definition}
\newtheorem{defn}[thm]{Definition}
\newtheorem{remark}[thm]{Remark}

\DeclareMathOperator{\C}{\mathbb{C}}

\DeclareMathOperator{\tr}{\textrm{trace}}

\begin{document}
\title{Local asymptotics for Hitchin's equations and high energy harmonic maps}
\author[Nathaniel Sagman]{Nathaniel Sagman}
\author[Peter Smillie]{Peter Smillie}
\begin{abstract}
We find new estimates and a new asymptotic decoupling phenomenon for solutions to Hitchin's self-duality equations at high energy. These generalize previous results for generically regular semisimple Higgs bundles to arbitrary Higgs bundles. We apply our estimates to the Hitchin WKB problem and to high energy harmonic maps to symmetric spaces and buildings. 

\end{abstract}
\maketitle

\section{Introduction} Let $S$ be a Riemann surface with canonical bundle $\mathcal{K}$. A Higgs bundle over $S$ is the data of a holomorphic vector bundle $(E,\overline{\partial}_E)$ over $S$ and a holomorphic section $\phi$ of $\textrm{End}(E)\otimes\mathcal{K}.$
A Hermitian metric $h$ on $(E,\overline{\partial}_E)$ determines a real structure $*_h$ on $\mathrm{End}(E)$ and a Chern connection $\nabla_h$ on $E$ with curvature form $F(h).$ We say that $h$ solves Hitchin's self-duality equations if 
\begin{equation}\label{1}
    F(h)+[\phi,\phi^{*_h}]=0.
\end{equation}
Equation (\ref{1}) expresses that the connection $$D_h=\nabla_h + \phi + \phi^{*_h}$$ is flat. Choosing a basepoint on $S$, we obtain a holonomy representation $\rho$ for $D_h,$ and the Hermitian metric is equivalent to a $\rho$-equivariant harmonic map from the universal cover $\tilde{S}$ to the Riemannian symmetric space $X_n$ of $\textrm{SL}(n,\C)$. For this reason, $h$ is called a harmonic metric, and $(E,\overline{\partial}_E,\phi,h)$ a harmonic bundle. There is an algebraic notion of stability for Higgs bundles, and the non-abelian Hodge correspondence, proved by Donaldson \cite{D}, Corlette \cite{C}, Hitchin \cite{Hi}, and Simpson, \cite{Si} says that if $(E,\overline{\partial}_E,\phi)$ is a stable Higgs bundle of degree $0$ on a closed surface $S$, then there is a unique metric $h$ solving (\ref{1}), and conversely any irreducible action of $\pi_1(S)$ on $X_n$ yields the data of an equivariant harmonic map to $X_n$ and a Higgs bundle.

Motivated by questions about hyper-K{\"a}hler geometry of the moduli space of Higgs bundles and its compactifications, $L^2$ cohomology, compactifications of character varieties, harmonic maps, WKB problems \cite{GMN}, \cite{KNPS}, \cite{MSWW}, etc., many authors have studied the behaviour of solutions to (\ref{1}) along rays of Higgs bundles $(E,\overline{\partial}_E,R\phi)$, $R>0,$ as we take $R\to \infty$. For example, in the paper \cite{SS} we made use of high energy estimates of Mochizuki \cite{Mo} to disprove the Labourie conjecture about the uniqueness of equivariant minimal immersions for Hitchin representations \cite{L1}, which was a central problem in higher Teichm{\"u}ller theory. 

Up until now, most results on this topic have applied to generically regular semisimple Higgs bundles, in the sense that the Higgs field has distinct eigenvalues almost everywhere. In this paper we find new estimates and results for arbitrary Higgs bundles. Our original motivation comes from minimal surfaces: in an upcoming paper, we will apply Theorem C of this paper to find more unstable minimal surfaces in symmetric spaces, expanding on the results of \cite{SS} and in particular disproving the Labourie conjecture for genus 2. In the end, our main results build on the theory of asymptotic decoupling for Hitchin's equations, in line with work of Mazzeo-Swoboda-Weiss-Witt \cite{MSWW}, Fredrickson \cite{Fre}, Mochizuki \cite{Mo}, and Mochizuki-Szab{\'o} \cite{MSz}, and we also contribute to the Hitchin WKB problem and the theory of high energy harmonic maps, relevant to works of Gaiotto-Moore-Neitzke \cite{GMN} and Katzarkov-Noll-Pandit-Simpson \cite{KNPS}.

We expect that our work here can be applied to study Hitchin-Simpson equations on higher dimensional K{\"a}hler manifolds (in special cases these become the Kapustin-Witten and Vafa-Witten equations). For work on the Hitchin-Simpson equations for large Higgs fields, see for instance,  \cite{Chen} and \cite{He}.

As well, we hope that our work will have further applications. For instance, it would be interesting to study the asymptotic geometry of the hyper-K{\"a}hler metric on the moduli space of Higgs bundles in non-generic directions, or to study solutions to Hitchin's equations on the disk for general Higgs bundles (see \cite{LMo}, \cite{MSz}, and \cite{Mohyp}). 

\subsection{Conventions}
When dealing with constants, we will write $C=C(x,y,z)$ to  mean that $C$ depends on quantities or objects $x,y,$ and $z$. A constant labelled $C$ in one result may differ from a constant labelled $C$ in another. Constants $C$ may even change in the course of a proof. We'll frequently replace bounds such as $Cde^{-cd}$, for $c>1,$ by $Ce^{-cd},$ absorbing the linear term into the exponential. For $r>0$, let $D(r)=\{z\in \C: |z|<r\}$. We use the flat metric on $D(r)$ to measure norms of differential forms. By conformal invariance, local analysis of Hitchin's equations on a Riemann surface can be carried out on $D(1)$. 

\subsection{Previous results on local asymptotics}  Let $\phi$ be a Higgs field on the trivial bundle $E=D(1)\times \C^n$. By holomorphicity, the dimensions of the generalized eigenspaces for $\phi$ are constant, apart from at a discrete set.
\begin{defn}\label{def: critical set}
    The critical set $B$ of $\phi$ is the discrete subset on which the number of generalized eigenspaces of $\phi$ is not maximal.
\end{defn}
Writing $\phi$ as a matrix-valued $1$-form $\phi(z)=f(z)dz$, the eigenvalue gap is 
$$\textrm{Gap}(\phi)(z) = \min\{ |\alpha(z) - \beta(z)|: \alpha(z),\beta(z) \textrm{ are distinct eigenvalues of }f\}.$$

\begin{defn} \label{dfn:S(d,A)}
Let $S_n(d,A)$ be the set of Higgs fields $\phi$ on the rank $n$ trivial bundle over $D(1)$ for which the critical set is empty, $\textrm{Gap}(\phi)\geq d$, and all eigenvalues of $f$ are bounded above by $Ad.$ 
\end{defn}

\begin{remark}[Scaling properties] If $\phi \in S_n(d,A)$, then $R\phi \in S_n(Rd,A)$. If $\phi$ is a Higgs field on the trivial bundle over $D(r)$ with $\mathrm{Gap}(\phi) \geq d$ and all eigenvalues are of modulus at most $Ad$, then $r^*\phi \in S_n(rd, A)$.
\end{remark}

We say that a splitting $E=\oplus_{i=1}^m E_i$ is $\epsilon$-almost orthogonal with respect to a metric $h$ (at least up to a constant - see Definition \ref{defn: almost orthogonal} and Lemma \ref{lem: almost orthogonal}) if for any sections $s_i$ of $E_i$ and $s_j$ of $E_j$, $i\neq j,$ we have $|h(s_i,s_j)|\leq \epsilon|s_i|_h|s_j|_h.$ Drawing on estimates of Simpson, Mochizuki proves the following.
\begin{thm}[Lemma 1.3 in \cite{Mo}]\label{GRSdecoupling} Let $\phi \in S_n(d,A)$, let $E=\oplus_{i=1}^m E_i$ denote the generalized eigenspace decomposition for $\phi$, and let $h$ be a harmonic metric for $\phi$. On any subdisk $D(r)$, the splitting $E = \oplus_{i=1}^m E_i$ is $Ce^{-cd}$-almost orthogonal for $h$ for constants $C=C(n, r, A)$ and $c=c(n,r,A)$.
\end{thm}
We remark that, unlike most results of \cite{Mo}, this theorem does not require that $\phi$ be regular semisimple. This theorem will play an important role in our paper. We also point out that Collier-Li prove sharper results on the asymptotics of harmonic metrics in the case of cyclic and sub-cyclic Higgs bundles in the Hitchin section \cite{CL}.

Recall that an endomorphism of a vector space is regular if the degree of its minimal polynomial is equal to that of its characteristic polynomial, and semisimple if it is diagonalizable; in particular, $\phi$ is regular semisimple if and only if each $E_i$ has dimension $1$. From Theorem \ref{GRSdecoupling}, Mochizuki deduces:
\begin{cor}[Theorem 1.2 in \cite{Mo}]\label{GRSasympdecoupling}
Assuming $\phi$ is regular semisimple, Hitchin's equations asymptotically decouple:
\begin{equation*}
    |F(h)|_{h} \leq Ce^{-cd} \textrm{ and } |[\phi,\phi^{*_{h}}]|_{h} \leq Ce^{-cd}.
\end{equation*}
\end{cor}
Let $S$ be a Riemann surface and $(E,\overline{\partial}_E,R\phi,h_R)$, $R>0$, be a ray of harmonic bundles on $S$, with $\phi$ generically regular semisimple. Coupled with Uhlenbeck compactness, Corollary \ref{GRSasympdecoupling} implies 

\begin{cor} [section 1.2.1 in \cite{Mo}] \label{cor:GRSlimit} As $R\to\infty,$ on the complement of the critical set, the quadruple $(E,\overline{\partial}_E,R\phi,h_R)$ sub-converges up to gauge transformations to a solution to the decoupled equations \begin{equation}\label{2}
   F(h)=0, \hspace{1mm} [\phi,\phi^{*_h}]=0.
\end{equation}
\end{cor}

Mazzeo-Swoboda-Weiss-Witt in \cite{MSWW} define limiting configurations on Riemann surfaces to be solutions to the decoupled Hitchin's equations (\ref{2}) over the complement of the critical set $B$ that also satisfy certain conditions at $B$. For closed Riemann surfaces, when $E$ has rank $2$ they prove that limiting configurations are limit points of the locus of Higgs bundles with smooth spectral curve. Still in rank $2$, in \cite{Mo} Mochizuki extends the result to generically regular semisimple Higgs bundles. See \cite{Fre} and \cite{MSz} for results on this subject for higher rank bundles.

\subsection{Asymptotics of solutions to Hitchin's equations}

We will prove analogues of Corollaries \ref{GRSasympdecoupling} and \ref{cor:GRSlimit} in the case where $\phi$ is not generically regular semisimple. These results are concerned only with the complement of the critical set; we hope to pursue the analysis at the critical set in future work.

Let $(E,\overline{\partial}_E,\phi)$ be a Higgs bundle on $S$ with critical set $B$. Our first observation is that $\phi$ admits a unique Jordan-Chevalley decomposition as the sum of commuting meromorphic Higgs fields as $$\phi=\phi_s+\phi_n,$$
where $\phi_s$ is generically semisimple and $\phi_n$ is nilpotent. These Higgs fields are holomorphic on $S-B,$ but might have poles on $B.$ To study arbitrary Higgs bundles, we will apply the analysis developed by Simpson and Mochizuki to $\phi_s$, and then treat $\phi_n$ separately. 

For our generalization of Corollary \ref{cor:GRSlimit}, we want to be slightly more precise about convergence up to gauge transformations: namely, we consider only gauge transformations that preserve the decomposition of $E$ into generalized eigenspaces of $\phi$. We write $\mathcal{G}^\oplus\subset \mathcal{G}$ for this subgroup of the gauge group of $E$; it can also be characterized as those gauge transformations commuting with $\phi_s$.

\begin{thma}[Asymptotic decoupling]
Let $S$ be an arbitrary Riemann surface and  $(E,\overline{\partial}_E,\phi)$ a Higgs bundle on $S$ with critical set $B$. Let $R_i$ be a sequence in $\R^+$ tending to infinity, and let $h_i$ be any sequence of harmonic metrics on $(E,\overline{\partial}_E,R_i\phi)$. 
Passing to a subsequence, there exist $s_i \in \mathcal{G}^\oplus$ such that the triples $(s_i^*\overline{\partial}_E,  s_i^* R_i \phi_n, s_i^* h_{R_i})$ converge in $C^\infty_{\mathrm{loc}}(S - B)$ to data $(\dbar_\infty, \psi_\infty, h_\infty)$ with $\psi_\infty$ nilpotent and satisfying the following decoupled version of Hitchin's equations with respect to $\phi_s$:
\begin{equation} \label{decoupled equations}
    F(h_\infty) + [\psi_\infty,\psi_\infty^{*_{h_\infty}}]=0, \hspace{1mm} [\phi_s,\phi_s^{*_{h_{\infty}}}] =[\phi_s,\psi_\infty^{*_{h_\infty}}]=0.
\end{equation}
\end{thma}

Theorem A and the equation (\ref{decoupled equations}) have a spectral interpretation. We define the (non-critical) spectral cover $\hat{S}$ of $(E,\overline{\partial}_E,\phi)$ to be the Riemann surface covering $S-B$ such that a point of $\hat{S}$ is a point $p$ of $S-B$ together with a choice of generalized eigenspace of $\phi$ at $p$.  Let $\pi: \hat{S} \to S-B$ be the natural projection. There is a tautological vector bundle $\hat{E}$ on $\hat{S}$ (possibly with fibers of different dimension on different components) whose pushforward $\pi_*\hat{E}$ is naturally isomorphic to $E|_{S-B}$, and every Higgs field $\psi$ on $S-B$ commuting with $\phi_s$ (such as $\phi$ itself) is the pushforward of a corresponding Higgs field $\hat{\psi}$ on $\hat{S}$. Then $\mathcal{G}^\oplus$ is the image of the gauge group of $\hat{E}$ in the gauge group of $E|_{S-B}$.
Note that if $\phi$ is generically regular semisimple, then $\hat{E}$ is a line bundle on $\hat{S}$ and $\hat{\phi}$ is a $1$-form on $\hat{S}$.

Solving the decoupled Hitchin's equations (\ref{decoupled equations}) is equivalent to the statement that $(\dbar_\infty, \psi_\infty, h_\infty)$ comes from the pushforward by $\pi$ of a nilpotent harmonic bundle structure on $\hat{E}$.

\begin{cora}
    In the setting of Theorem A, the sequence $(\dbar_E, R_i\phi_n,h_R)$ sub-converges up to gauge to the pushforward of a nilpotent harmonic bundle structure on the the tautological bundle $\hat{E}$ over the spectral cover $\hat{S}$.
\end{cora}

The proof of Theorem A will use the estimates from Theorems B and C below on local asymptotics for solutions to Hitchin's equations. Together they give a uniform curvature bound that allows us to apply an Uhlenbeck compactness argument in the proof of Theorem A. We also apply Theorems B and C to the analysis of high energy harmonic maps.

Theorem B is the main analytic input; applying a method of Simpson and Mochizuki, we give an a priori bound for the norm of the nilpotent part of $\phi$. We remark that in the special case that $\phi$ is already nilpotent, it reduces to an estimate of Simpson (see \cite[section 2]{S} and also Proposition \ref{firstbound}).
\begin{thmb}\label{mainestimate} Let $\phi \in S_n(d,A)$ and let $h$ be a harmonic metric for $\phi$. On any subdisk $D(r)\subset D(1)$, there exists $C=C(n,r,A)>0$ such that
$$|\phi_n|_h\leq C.$$
\end{thmb}
The point of this theorem is that $|\phi_n|_h$ is bounded independently of $d,$ and hence the bound is stable along rays of harmonic bundles $(E,\overline{\partial}_E,R\phi,h_R).$ 

Theorem C is a consequence of Theorem B together with the argument of \cite[Proposition 2.11]{Mo}. Keeping with the objects of Theorem B, let $E=\oplus_{i=1}^m E_i$ be the generalized eigenspace decomposition of $E$ with respect to $\phi$. We let $B_i$ be the second fundamental form of $E_i$ in $E$, meaning $B_i(s) = (\nabla s)^\perp$ for $s$ a section of $E_i$, where the perpendicular is taken with respect to $h$. We say that the sub-bundle $E_i$ is $\epsilon$-almost parallel with respect to $h$ if $|B_i|_h \leq \epsilon$.

\begin{thmc}
 Let $\phi \in S_n(d,A)$ and let $h$ be a harmonic metric for $\phi$. 
    On any subdisk $D(r)\subset D(1)$, there exist constants $C=C(n,r,A), c=c(n,r,A)>0$ such that each sub-bundle $E_i$ is $Ce^{-cd}$-almost parallel for $h$.
\end{thmc}
When $\phi$ is generically regular semisimple, Theorem C follows from \cite[Proposition 2.11]{Mo}. Our work removes the generically regular semisimple assumption. 

Using Theorem C, we can write an asymptotic expression for the flat connection. This will be used in the proof of Theorem D. Let $D_h = \nabla_h + \phi+\phi^{*_h}$ be the flat connection associated with our harmonic bundle and let $\phi_i$ be the eigen-$1$-forms of $\phi$, i.e. $\phi_s=\sum_{i=1}^m \phi_i\pi_i.$ Set $h_i$ to be the restriction of $h$ to $E_i$ and $h^{\oplus}=\oplus_{i=1}^m h_i.$ Note that $h^{\oplus}$ is the pushforward of the lift of $h$ to the spectral cover. We denote the Chern connection of $h^{\oplus}$ by $\nabla_h^{\oplus}.$
\begin{corc} In the setting of Theorem C, for some $\textrm{End}(E)$-valued $1$-form $a$ satisfying $|a|_{h}\leq Ce^{-cd},$
$$D_h = \nabla^\oplus_h +  \sum_{i=1}^m (\phi_i+\overline{\phi}_i)\pi_i+ \phi_n +\phi_n^{*_{h^\oplus}}+a.$$
\end{corc}

\subsection{The Hitchin WKB problem}\label{HWKBsection}
We can apply our local analysis to the Hitchin WKB problem. Let $S$ be a Riemann surface and $(E,\overline{\partial}_E,R\phi, h_R)$, $R>0$, a ray of harmonic bundles on $S$. The Hitchin WKB problem posed in \cite{GMN} and \cite{KNPS} asks about the monodromy of the flat connection $D_{h_R}$ as we take $R\to \infty.$ 

It is formulated as follows. Let $V$ be a complex vector space of dimension $n$, and let $h_1,h_2$ be unit volume Hermitian metrics on $V.$ Simultaneously diagonalizing $h_1$ and $h_2,$ we obtain a basis $\{e_1,\dots, e_r\}$ of $V$ that is orthogonal for both metrics. Define $\kappa_i = \log |e_i|_{h_1}-\log|e_i|_{h_2},$ and reorder our bases so that $\kappa_1\geq \kappa_2\geq \dots \geq \kappa_r.$ The vector distance between $h_1$ and $h_2$ is defined by $$\vec{d}(h_1,h_2)=(\kappa_1,\dots, \kappa_r).$$ 
$\vec{d}(h_1,h_2)$ is also the Weyl chamber-valued distance between the points $h_1$ and $h_2$ in the symmetric space of $\textrm{SL}(V).$

Let $\gamma$ be a path in the complement of the critical set, and let $\phi_1, \ldots, \phi_n$ be an ordering of the eigen-1-forms of $\phi_s$ along $\gamma$, which might not be distinct. We assume that $\gamma$ is non-critical, which means that unless $\phi_i = \phi_j$ on all of $\gamma$, $\mathrm{Re}(\phi_i(\dot{\gamma}))$ is never equal to $\mathrm{Re}(\phi_j(\dot{\gamma}))$. We set $\alpha_i=-\int_\gamma \textrm{Re}(\phi_i),$ and reorder if necessary so that $\alpha_1\geq\alpha_2\geq \dots \geq \alpha_n.$

For each $R>0,$  let $\Pi_{R,\gamma}: E_{\gamma(0)}\to E_{\gamma(1)}$ be the parallel transport operator induced by the flat connection $D_{h_R}$. The following is conjectured in \cite{KNPS} for $\phi$ generically regular semisimple and proved in \cite[Theorem 1.5]{Mo} (for $\phi$ generically regular semisimple). Theorem $D$ removes hypotheses on $\phi$, with a slightly weaker conclusion in the case that $\phi$ is not semisimple.

\begin{thmd} Let $(E, \dbar_E, R\phi, h_R)$ be a family of harmonic bundles on a Riemann surface $S$, and let $\gamma$ be a non-critical path. Then
   $$|\frac{1}{R}\vec{d}((h_R)_{\gamma(0)},\Pi_{R,\gamma}^*(h_R)_{\gamma(1)})-2(\alpha_1,\dots, \alpha_n)|\leq CR^{-1},$$
   for some $C$ depending on $S$, $\gamma$, $n$, and the eigen-1-forms of $\phi$. If $\phi$ is generically semisimple, then the bound can be improved to $C e^{-cR}$ for some $c$ with the same dependence.
\end{thmd}
Using Theorem B and Corollary C, one can prove Theorem D by following Mochizuki's proof of Theorem 1.5 in \cite{Mo} almost verbatim. For this reason, we've put the proof in Appendix A.

In \cite{LTW}, for cyclic $\textrm{SL}(3,\R)$-Higgs bundles in the Hitchin component, the authors study paths that go through the critical set. In the general case, there is an expectation that critical paths can be understood via spectral networks (see \cite[section 3.4]{KNPS}).

\subsection{Application to harmonic maps}
The last of our main results concerns high energy harmonic maps to symmetric spaces of non-compact type. Our results so far suggest that, for $\phi \in S_n(d,A)$ and $d$ large, every solution to the self-duality equations gives a harmonic map to the $\textrm{SL}(n,\C)$-symmetric space that is approximated well in a certain sense by the harmonic functions obtained by integrating the real parts of the eigen-$1$-forms of $\phi$. This suggestion is made precise and generalized by Theorem E. The generalization is two-fold: we work with harmonic maps to arbitrary symmetric spaces of non-compact type, and instead of working locally on the disk, we construct a comparison to a global object, namely a harmonic map from a small cameral cover to a split toral subalgebra. 

In order to state Theorem E, we must first introduce $G$-Higgs bundles and the Hitchin base. Fix a semisimple Lie group $G$ with no compact factors, and maximal compact subgroup $K$, as well as a $G$-invariant Riemannian metric $\nu$ on $G/K,$ so that $(G/K,\nu)$ is a symmetric space of non-compact type. Let $\mathfrak{g}$ and $\mathfrak{k}$ be the Lie algebras of $G$ and $K$ respectively, and let $\mathfrak{p} \subset \mathfrak{g}$ be the orthogonal complement to $\mathfrak{k}$. 

A $G$-Higgs bundle is a holomorphic principal-$K^{\C}$ bundle $P$ together with a holomorphic $\mathrm{ad}\p^{\C} := P\times_{\textrm{Ad}|_{K^{\C}}} \p^{\C}$-valued $1$-form $\phi$, called the Higgs field. A reduction $h$ of the structure group of $P$ from $K^{\C}$ to $K$ determines a real structure $*_h$ on $\textrm{ad}\p^{\C}$ as well as a Chern connection $A_h$ on $P$, and hence Hitchin's equations \eqref{1} make sense as an identity of two-forms valued in the adjoint bundle of $P$. A solution to Hitchin's equations for a $G$-Higgs bundle $(P,\phi)$ on $S$ is the same, up to equivalence, as an equivariant harmonic map from $\tilde{S}$ to $(G/K,\nu).$ The data $(P,\phi,*_h)$ is called a harmonic $G$-bundle.

Fix additionally a maximal split toral subalgebra $\ft\subset \g$ with (restricted) Weyl group $W$. We write $\mathcal{K}$ for (the total space of) the canonical bundle of the surface $S$, and $\mathcal{K} \otimes \ft^{\C}$ for the bundle of $\ft^{\C}$-valued one forms. We write $\mathcal{K} \otimes \ft^{\C} // W$ for the invariant-theoretic quotient of $\mathcal{K} \otimes \ft^{\C}$ by the action of $W$ on $\ft^{\C}$, which is equivalently a quotient of $\mathcal{K} \times \ft^{\C}$ by $\C^* \times W$. It is a fiber bundle over $S$, whose fibers are noncanonically isomorphic to the cone $\ft^{\C}//W$. We refer to a holomorphic section $\eta$ of $\mathcal{K} \otimes \ft^{\C} // W$ as a $\ft^{\C} // W$-valued $1$-form on $S$. A $\ft^{\C} // W$-valued $1$-form is equivalent to a point in the well-known Hitchin base (see section \ref{HBsection}).

A $\ft^{\C} // W$-valued $1$-form is also equivalent to a cameral covering in the sense of \cite[section 2]{Don}. In one direction, a $\ft^{\C}//W$-valued $1$-form $\eta$ determines by pullback a scheme $\mathrm{Cam}$ with a $W$ action covering $S$, together with a $W$-equivariant $\ft^{\C}$-valued $1$-form $\eta_{\mathrm{Cam}}$ on $\mathrm{Cam}$. The map in the other direction is just the $W$ quotient.

In section \ref{sec: cameral covers} we define the notion of a small cameral cover, which is a connected Riemann surface with a $\ft^{\C}$-valued $1$-form, which is a Galois covering of $S-B$ with Deck group a subquotient of $W$. If $G = \mathrm{SL}(n,\C)$, a point in a small cameral cover lying over a point $p \in S-B$ can be thought of as an ordering of the generalized eigenspaces of the Higgs field at $p$ (see \cite[section 3.2]{SS}). Fixing a small cameral cover $C$ with $\ft_{\C}$-valued $1$-form $\eta_C$ and Deck group $W_C$, integrating $\mathrm{Re}(\eta_C)$ gives a harmonic map $f_\eta$ from the universal cover $\tilde{C}$ of $C$ to $\ft$ that is equivariant by an action of the semi-direct product of $W_C$ with $\pi_1(C)$ on $\ft$. We call $f_\eta$ the toral map. 

The $G$-invariant Riemannian metric $\nu$ on $G/K$ determines by restriction a Euclidean metric $m$ on $\ft$. While the map $f_\eta$ depends on some choices, its derivatives do not; in particular, the pullback metric $f_\eta^*m$ descends canonically to a symmetric $2$-tensor on $S$.
 
Assume that for each $R>0$, the Hitchin equation for $(P,R\phi)$ admits a solution, corresponding to a harmonic map $f_R$ from the universal cover of $S$ to $(G/K, \nu)$. By the equivariance of $f_R$, the pullback metric $f_R^*\nu$ also descends to a metric on $S$. Let $\eta$ be the $\ft^{\C}//W$-valued $1$-form associated with $(P,\phi)$. The first part of Theorem E states that these metrics converge to $f_\eta^*m$ as $R$ tends to infinity.

The second part of Theorem E involves a comparison between the second fundamental forms of $f_\eta$ and $f_R$. The set-up for such a comparison is a little bit involved. The pullback $f_R^*TG/K$ of the tangent bundle descends to a bundle on $S$ that is the real form of $\mathrm{ad}\p^{\C}$ determined by the anti-involution $*_h$. A certain flat sub-bundle of $f_\eta^*T(\ft)$ descends to a flat bundle $\R F_\eta$ on $S - B$ with monodromy in the group $W_C$. There is a natural inclusion of $\R F_\eta$ into $\textrm{ad}\p^{\C}$, and composing this inclusion with the projection of $\textrm{ad}\p^{\C}$ to the self-adjoint (i.e., real) endomorphisms gives a mapping of bundles $\textrm{Re}_R: \R F_\eta\to f_R^*TG/K$. Our last result involves a comparison of the pullback metrics as well as the Levi-Civita connections $\nabla^{f_\eta^*}$ and $\nabla^{f_R^*}$ along this sub-bundle. It says that the map $\textrm{Re}_R$ is nearly parallel.

Fix a metric $\sigma$ on $S$. 
\begin{thme} Let $(P,R\phi,*_{h_R})$, $R>0,$ be a family of harmonic $G$-bundles, and let $f_R: \tilde{S} \to (G/K, \nu)$ and $f_\eta: \tilde{C} \to (\ft, m)$ be the harmonic maps and toral maps respectively. Let $U\subset S-B$ be a relatively compact open subset. There exists $C=C(G,U,\sigma, \eta)$ and $c=c(G,U,\sigma, \eta)$ such that, on $U,$
\begin{equation}\label{pbasymptotic}
        |R^{-2}f_R^*\nu(z) - f_\eta^*m(z)|_\sigma \leq CR^{-2},
    \end{equation}
and if $\phi$ is generically semisimple, then the bound can be improved to $Ce^{-cR}$.

Moreover, if $\mathrm{Re}_R$ is the map of tangent bundles defined above and $\nabla_R$ the natural connection on $\mathrm{Hom}(\R F_\eta, f_R^*TG/K)$, then on $U,$ 
\begin{equation}\label{iotanearlyparallel}
    |\nabla_R \mathrm{Re}_R|_{h_R,\sigma} \leq Ce^{-cR}.
\end{equation}
%The dependence of $C_j$ and $c$ on $U$ is on the order of $d_\sigma(U,B)^{-2},$ and 
\end{thme}
For harmonic maps between hyperbolic surfaces, i.e., the case $G=\textrm{PSL}(2,\R)$, estimates like (\ref{pbasymptotic}) and (\ref{iotanearlyparallel}) originate from the thesis work of Wolf \cite{Wo} and Minsky \cite{Mi}. In that case, the critical set is the zero set of the Hopf different of the harmonic map. Geometrically, for $G=\textrm{PSL}(2,\R),$ the harmonic maps converge to a harmonic map to an $\R$-tree. In higher rank, high energy harmonic maps approximate harmonic maps to affine buildings.

Continuing in the setting above, upon choosing a non-principal ultrafilter $\omega$ and a basepoint $p$ in $\tilde{S}$, by Kleiner-Leeb \cite{KL}, for any sequence $R_i$ increasing to infinity, the asymptotic cone of the rescaled pointed symmetric spaces $(G/K,f_{R_i}(p),R_i^{-1}\nu)$ is an affine building $(\mathcal{B},d)$. Using now routine estimates, we prove that as $R_i\to \infty,$ the harmonic maps $f_{R_i}$ $\omega$-converge to a map $f_\omega:\tilde{S}\to (\mathcal{B},d)$ that is harmonic in the sense of Korevaar-Schoen \cite{KS} and equivariant with respect to a representation $\rho_\omega:\pi_1(S)\to \textrm{Isom}(\mathcal{B},d)$. Harmonic maps to buildings admit an $L^1$ measurable pullback metric in the sense of Korevaar-Schoen (see \cite{KS}) and as well give the data of a $\ft^{\C}//W$-valued holomorphic $1$-form (see section \ref{sec: harmonic to building}). We denote the $L^1$ measurable pullback metric of $f_\omega$ (in the sense of \cite{KS}) by $g_{f_\omega}$.
\begin{core}
Let $f_\omega: \tilde{S}\to (\mathcal{B},d)$ be the harmonic map to the building obtained by taking the $\omega$-limit of the harmonic maps $f_{R_i}:\tilde{S}\to (G/K,f_{R_i}(p),R_i^{-1}\nu)$ and let $f_\eta:\tilde{C}\to (\ft,m)$ be the toral map.
\begin{enumerate}
\item As measurable tensors on $S$, $g_{f_\omega}=f_\eta^*m$.
    \item The $\ft^{\C}//W$-valued holomorphic $1$-form of $f_\omega$ is equal to $\eta$.
\end{enumerate}
\end{core}
The geometric meaning of (2) is that if we look at a contractible region $U$ mapping into a single apartment chart for the building, then $f_\omega$ and $f_\eta$ (on a lift of $U$) differ by a translation and the action of an element of $W$.
In other words, all information about $f_\omega$ that does not involve its behaviour at critical set is completely encoded by the $\ft^{\C}//W$-valued $1$-form $\eta$. Moreover, the limiting map to the building does not see the nilpotent part of the Higgs field. 
\begin{remark}
    Theorem D can be interpreted in terms of $\ft^{\C}//W$-valued $1$-forms and harmonic maps to buildings. For a non-critical path $\gamma$ and the harmonic maps $f_R,$ Theorem D says that the Weyl-chamber valued distance between $f_R(\gamma(0))$ and $f_R(\gamma(1))$ rescaled by $R^{-1}$ converges to the Weyl-chamber valued distance between $f_\omega(\tilde{\gamma}(0))$ and $f_\omega(\tilde{\gamma}(1)),$ where $\tilde{\gamma}$ is a lift of $\gamma$ to $\tilde{C}$.
\end{remark}

\begin{remark}
    In \cite{Ma}, Martone uses Fock-Goncharov coordinates to construct diverging sequences of Hitchin representations that are ``tree-like" at infinity. See also work of Burger, Iozzi, Parreau, and Pozzetti \cite{BIPP}. If we choose $\eta$ so that $\textrm{Re}(\eta)$ always lives in a one-dimensional subspace of $\ft,$ then the image of the map $f_\omega$ will be $1$-dimensional. Thus, using the Hitchin section, we can construct diverging sequences of (non-Fuchsian) Hitchin representations whose harmonic maps sub-converge to a map with $1$-dimensional image.
\end{remark}

\subsection{Outline of paper}
In the next section we give preliminaries on linear algebra that will be useful for studying Higgs bundles and the local analysis of Hitchin's equations. In section 3 we discuss constructions involving Higgs bundles such as their Jordan-Chevalley and Schur decompositons, as well as some consequences of Mochizuki's work in \cite{Mo}.

In section 4 we prove Theorem B and in section 5 we prove Theorem C and Corollary C and derive further estimates that will come into play in the proof of Theorem A. Then in section 6 we prove Theorem A, making use of both Theorem B and Theorem C.

Sections 7 and 8 are devoted to harmonic maps. In section 7 we give the preliminary theory we will need on harmonic maps to symmetric spaces, $G$-Higgs bundles, harmonic maps to buildings, and $\ft^{\C}//W$-valued $1$-forms. Then in section 8 we prove Theorem E using the local estimates that we found for Higgs bundles in Theorems B and C, and we prove Corollary E using Theorem D. 

The proof of Theorem D is given in Appendix A.

\subsection{Acknowledgements}
Nathaniel Sagman was funded by the Fonds National de la Recherche grant O20/14766753, \textit{Convex Surfaces in Hyperbolic Geometry.}

\section{Linear algebraic preliminaries}
Throughout this section, let $V$ be a complex vector space of dimension $r$ with Hermitian metric $h$ and associated norm $|\cdot|_h$, and let $f:V\to V$ be an endomorphism. Let $*$ be the adjoint with respect to $h$. We also use $|\cdot|_h$ to denote the Hilbert-Schmidt norm on $\mathrm{End}(V)$ with respect to $h$.

\subsection{The Schur decomposition}\label{schurdec} A Schur basis for $f$ is an orthonormal basis of $V$ with respect to which $f$ is weakly upper triangular. In such a basis we can split
\[
f = f_a + f_u
\]
into its diagonal part $f_a$ and strictly upper triangular part $f_u$. Such a decomposition of $f$ with respect to any Schur basis is called a Schur decomposition. 

One construction is as follows: pick an eigenvalue $\lambda\in \C$ for $f$ with corresponding eigenspace $V_\lambda$ (not the whole generalized eigenspace), and choose an $h$-orthonormal basis for $V_\lambda$. If $V_\lambda^\perp$ is the $h$-orthogonal complement, then with respect to the splitting $V=V_\lambda\oplus V_\lambda^\perp,$ the matrix of $f$ splits into block matrices $$f=\begin{pmatrix} \lambda I & A \\
0 & B\end{pmatrix}.$$ The matrix has the right form, except for possibly the $B$ block. We restrict to $V_\lambda^\perp$ and then inductively repeat the procedure above. 

The Schur decomposition $f = f_a + f_u$ is uniquely determined by an ordering of the generalized eigenspaces of $f$ (though the Schur basis itself has some further indeterminacy). It can also be seen from the construction that any family of commuting matrices share a common Schur basis. Note that with respect to any Schur decomposition, $|f_a|_h^2$ is the sum of the squares of the eigenvalues of $f$ with multiplicity. Consequently both the functions $|f_a|^2_h$ and $|f_u|^2_h = |f|^2_h - |f_a|^2_h$ are independent of the choice of Schur decomposition.

In view of Simpson's inequality (see Lemma \ref{lem: origsim} below), the following estimate will be useful.
      \begin{prop}[Page 729 in \cite{S} and Lemma 3.5 in \cite{LM1}]\label{SMprop}
        There exists a constant $C=C(n)>0$ such that 
        \begin{equation}\label{SM}
            |f_u|_h^2\leq C|([f,f^{*_h}])|_h.
        \end{equation}
    \end{prop}

\subsection{Jordan-Chevalley and Schur} 

The Jordan-Chevalley decomposition (\cite[section 4.2]{Hu}) is the unique additive decomposition of $f$ into commuting semisimple and nilpotent parts: $f = f_s + f_n$. Fixing an ordering of the generalized eigenspaces of $f$, we can first apply the Jordan decomposition and then (since $f_n$ and $f_s$ commute) a simultaneous Schur decomposition with respect to a compatible Schur basis:
\[
f = f_s + f_n = (f_s)_a + (f_s)_u + (f_n)_a + (f_n)_u.
\]
Since $f$ is weakly upper triangular in the compatible Schur basis, $f_n$ will be strictly upper triangular, so $(f_n)_a = 0$ and $(f_s)_a = f_a$. Hence we obtain the Jordan-Chevalley-Schur decomposition:
\[
f = f_a + (f_s)_u + f_n. 
\]
In the Schur basis, $f_a$ is the diagonal part of $f$, $f_n$ is supported in the small strictly-upper triangles of each Jordan block, and $(f_s)_u$ generally has nonzero entries everywhere above the diagonal.

\subsection{Centres of centralizers}\label{DCthm}
The centralizer of $f$ is $C(f)=\{g\in \textrm{End}(V): [f,g]=0\},$ and the centre of the centralizer is $$ZC(f)=\{s\in \textrm{End}(V): [s,g]=0 \textrm{ for all } g\in C(f)\},$$ which is a subalgebra of $\textrm{End}(V).$ In fact it is the subalgebra generated by $f$, though we will not need this characterization.
It will be helpful to have an explicit basis for the centre of the centralizer of the semisimple part $f_s$. Let $V=\oplus_{i=1}^m V_i,$ be the generalized eigenspace decomposition for $f$, or equivalently the eigenspace decomposition for the semisimple part, $ZC(f_s).$ For each $i,$ let $\pi_i:V\to V_i$ be the projection with respect to the decomposition above.
\begin{prop}\label{pispan}
    $ZC(f_s)$ is spanned by the projections $\pi_1,\dots, \pi_m.$
\end{prop}
\begin{proof}
    Since $f_s$ is acting on each $E_i$ by multiplication by a complex scalar, any $g\in C(f)$ is of the form $g=\sum_{i=1}^m g_i\pi_i,$ where $g_i$ is some endomorphism of $E_i$. Therefore, the centre of the centralizer consists of endomorphisms of the form $g=\sum_{i=1}^n a_i\pi_i,$ where $a_i\in \C.$
\end{proof}

\begin{prop}\label{note}
   $ZC(f_s)\subset C(f) \subset C(f_n).$
\end{prop}
\begin{proof}
    For the first inclusion, since $f_s$ commutes with $f$, if $[s,g] = 0$ for all $g \in C(f_s)$, then in particular $[s,f] = 0$. The second inclusion holds because $f_n$ can be expressed as a polynomial in $f$ (see, e.g., \cite[section 4.2]{Hu}).
\end{proof}
The algebra $ZC(f_s)$ comes with a real structure $\dagger$, defined by the condition that the $\pi_i$'s are real: if $s=\sum_{i=1}^m s_i\pi_i\in ZC(f)$, then $$s^\dagger=\sum_{i=1}^m \overline{s_i}\pi_i.$$
\begin{defn} \label{defn: real toral}
    The real centre of the centralizer $\mathbb{R}ZC(f_s)$ is the real subalgebra consisting of the fixed point of $\dagger.$
\end{defn}
By the proposition above, $\mathbb{R}ZC(f_s).$ is the real span  of the $\pi_i$'s. Any choice of non-degenerate Killing form on the Lie algebra $\mathrm{End}(V)$ (for instance, $\langle A, B \rangle = \mathrm{tr}(AB)$) restricts to a positive definite metric on $\mathbb{R}ZC(f_s)$.

\subsection{Almost orthogonality}
Let $V$ be a complex vector space with Hermitian form $h$, and let $V = \oplus_{i=1}^m V_i$ be a decomposition of $V$ into subspaces. Let $\pi_i$ be the projection onto $V_i$ with kernel $\oplus_{j \neq i} V_j$, and let $\pi_i'$ be the orthogonal projection onto $V_i$. Extrapolating from \cite{Mo}, we make the following definition.
\begin{defn} \label{defn: almost orthogonal}
    A decomposition $V = \oplus_{i=1}^m V_i$ is $\epsilon$-almost orthogonal if for each $i$, $|\pi_i - \pi_i'|_h \leq \epsilon\; \mathrm{rank}(V_i)$.
\end{defn}

We now give some equivalent characterizations of almost orthogonality. Recall that for any endomorphism $f$, the quantity $|f_u|_h$ is independent of the Schur decomposition $f = f_a + f_u$. Letting $\pi$ stand for any $\pi_i$, the lemma below gives equivalent definitions of almost orthogonality.

\begin{lem}\label{lem:almostorthogonal1} The following equalities hold:
\begin{enumerate}
    \item $|\pi - \pi'|_h = |\pi_u|_h$
    \item $|\pi - \pi'|_h = \frac{1}{\sqrt{2}} |\pi - \pi^*|_h$.
\end{enumerate}
\end{lem}

\begin{proof}
    To prove (1), take the Schur decomposition $\pi = \pi_a + \pi_u$ with respect to the ordering that puts the 1-eigenspace first and the 0-eigenspace second. Then $\pi_a$ is exactly $\pi'$, so $\pi - \pi' = \pi_u$. 
    
    To prove (2), we first calculate
    \[
    |\pi - \pi^*|_h^2 = \tr((\pi - \pi^*)(\pi - \pi^*)^*) = 2\tr(\pi\pi^*) - \tr(\pi^2) - \tr((\pi^*)^2) = 2 |\pi|^2_h - 2\tr(\pi).
    \]
   On the other hand, since diagonal matrices are $h$-orthogonal to upper triangular matrices, we have
    \[
    |\pi|^2_h = |\pi_a|_h^2 + |\pi_u|_h^2.
    \]
    Together with $|\pi_a|^2_h = \tr(\pi_a) = \tr(\pi)$, this implies (2).   
\end{proof}

\begin{lem} \label{lem: almost orthogonal} If $V = \oplus_{i=1}^m V_i$ is $\epsilon$-almost orthogonal, then 
\begin{equation}\label{eqn: almost orthogonal}
h(v_i, v_j) \leq \sqrt{2}\epsilon |v_i|_h|v_j|_h
\end{equation}
for any vectors $v_i$ in $V_i$ and $v_j$ in $V_j$, $i \neq j$. Conversely, there are constant $C(n)$, $\epsilon(n)$ such that if \eqref{eqn: almost orthogonal} holds with $\epsilon < \epsilon(n)$, then the decomposition is $C\epsilon$-almost orthogonal.
\end{lem}

\begin{proof}
    For any $v_i\in V_i$ and $v_j\in V_j,$ since $\pi_iv_i=v_i$ and $\pi_i v_j =0,$ $$h(v_i,v_j)  = h(\pi_i v_i,v_j) - h(v_i, \pi_i v_j)=h((\pi_i-\pi_i^*)v_i,v_j)\leq \sqrt{2}\epsilon |v_i|_h |v_j|_h.$$

    For the converse, we give a more informal justification. For each $V_i,$ we choose an $h$-orthonormal basis. Assembling these bases into a basis of $V,$ the matrix $H$ representing $h$ has identity block matrices corresponding to each $V_i,$ and (\ref{eqn: almost orthogonal}) implies that the norms of the entries outside the blocks are bounded above by $\epsilon.$ The matrix $\Pi_i$ of $\pi_i$ is the identity on a block corresponding to $V_i,$ and zero elsewhere, and it is easily computed that $\Pi_i^*=H^{-1}\Pi_i H$ is $\epsilon$-close to $\Pi_i.$
\end{proof}

\begin{lem} \label{lem: star and dagger}
If $V = \oplus_{i=1}^m V_i$ is $\epsilon$-almost orthogonal, and $s = \sum_i s_i \pi_i$, then 
\[
|s^* - s^\dagger|_h \leq \sqrt{2n} \epsilon |s_a|_h
\]
and
\[
|s_u|_h \leq \sqrt{n}\epsilon |s_a|_h.
\]
\end{lem}

\begin{proof} Using the Cauchy-Schwarz inequality and Lemma \ref{lem:almostorthogonal1} (2), we have
\[
    |s^* - s^\dagger|_h^2 = |\sum_i \overline{s}_i(\pi_i^* - \pi_i)|^2_h \leq n\sum_i |s_i|^2|\pi_i^* - \pi_i|^2_h \leq 2n\epsilon^2 \sum_i |s_i|^2 \tr(\pi_i) = 2n\epsilon^2 |s_a|^2.
\]
For the second inequality, we observe that since the $\pi_i$ commute, there is a simultaneous Schur basis for all of them, such that if $s = \sum_i s_i \pi_i$ then 
\[
s_u = \sum s_i (\pi_i)_u.
\]
Then one uses the Cauchy-Schwarz and Lemma \ref{lem:almostorthogonal1} (1) as above.
\end{proof}

\section{Higgs bundles}\label{prelimhiggs}

Throughout, let $S$ be a Riemann surface with canonical bundle $\mathcal{K}.$ 

\subsection{Higgs bundles and Hitchin's equations}\label{Higgsintro}
Here we'll give the basic defintions, more formally than in the introduction.
\begin{defn}
A Higgs bundle of rank $n$ on $S$ is the data $(E,\overline{\partial}_E,\phi),$ where $(E,\overline{\partial}_E)$ is a holomorphic vector bundle of rank $n$ and $\phi$ is a holomorphic section of $\textrm{End}(E)\otimes \mathcal{K}$ called the Higgs field.
\end{defn}
For simplicity, we will always assume that $\deg E=0$ when $S$ is closed. Given a Hermitian metric $h$ on $E$, we use $*_h$ for the adjoint operator on $\textrm{End}(V)$ determined by $h$, $\nabla_h$ for the Chern connection, and $F(h)$ for the curvature of the Chern connection. We extend $h$ to $E$-valued forms $\Omega^k(E)$ and $\textrm{End}(E)$-valued forms by pairing sections in $E$ and wedging forms as usual. The adjoint $*_h$ then extends to $\textrm{End}(E)$-valued forms. Moving toward Hitchin's equations, we are looking to find a metric $h$ satisfying (\ref{1}), $$F(h) + [\phi,\phi^{*_h}]=0,$$ and we refer to quadruples $(E,\overline{\partial}_E,\phi,h)$ as harmonic bundles. Well-known existence theorems for solutions to (\ref{1}) on closed surfaces can be found in \cite{Hi} and \cite{Si}. A complex gauge transformation is a $C^\infty$ automorphism of the bundle $E$ that covers the identity map from $S\to S$. The group of gauge transformations $\mathcal{G}$ acts on sections of $E$ and sections of associated bundles, like connections and curvature forms. Standard formulas for these actions can be found in \cite[Chapter 2]{DK}, and it is easily seen that gauge transformations preserve the equation (\ref{1}).

\subsection{The Jordan-Chevalley decomposition of a Higgs bundle} 
Let $(E,\overline{\partial}_E,\phi)$ be a Higgs bundle of rank $n$ over $S$. Locally, we can trivialize $(E,\overline{\partial}_E)$ to the trivial bundle of dimension $n$ over $D(r)$, with coordinate $z$. In the trivialization, $\phi(z)=f(z)dz,$ for some holomorphic section of $\textrm{End}(\C^n)\times D(r)$ over $D(r)$.

\begin{defn}The Higgs field $\phi$ is said to be \emph{generically} (regular) semisimple if in any such local trivialization, the endomorphism $f(z)$ is (regular) semisimple on the complement of a discrete set.
\end{defn}

Alternatively, $\phi$ is generically (regular) semisimple if and only if for any meromorphic vector field $V$ on $S$, $\phi(V)$ is (regular) semisimple when viewed as an endomorphism of the vector space of meromorphic sections of $E$ over the field $F$ of meromorphic functions on $S$.

Extending Definition \ref{def: critical set} to a Riemann surface, the critical set $B$ of $\phi$ is the discrete set on which the number of generalized eigenspaces of $\phi$ is not maximal. Concretely, these are the points at which two locally defined different eigen-$1$-forms of $\phi$ match up. For $\phi$ generically regular semisimple, the critical set agrees with the discriminant locus from \cite{Mo}.
\begin{prop}\label{JCdec}
There is a unique Jordan-Chevalley decomposition
$$\phi = \phi_s + \phi_n
$$
of the Higgs field $\phi$ into a sum of commuting meromorphic Higgs fields $\phi_s$ and $\phi_n$ such that $\phi_s$ is generically semisimple and $\phi_n$ is nilpotent. Away from the critical set, $\phi_s$ and $\phi_n$ are holomorphic and agree with the pointwise Jordan-Chevalley decomposition.
\end{prop}

\begin{proof}
For any meromorphic vector field $V$, apply the Jordan-Chevalley decomposition over the field $F$ of meromorphic functions on $S$ to the endomorphism $\phi(V)$ to give a decomposition of endomorphisms $\phi(V) = \phi^V_{s} + \phi^V_n$. Since the decomposition is linear over $F$, $\phi^V_{s} = \phi_s(V)$ for a possibly meromorphic 1-form $\phi_s$, and similarly for $\phi^V_n$.

If $U$ is a contractible open subset of $S$ disjoint from the critical set, then by the canonicity of the Jordan-Chevalley decomposition, $\phi_s$ restricted to $U$ is the semisimple part of $\phi$ restricted to $U$. On the other hand we can decompose $E = \oplus_{i=1}^m E_i$ into generalized eigenspaces of $\phi$ over $U$, with corresponding eigen-$1$-forms $\alpha_i$ and set $\phi_s^U = \sum \alpha_i \mathrm{Id}_{E_i}$. Since $\phi_s^U$ is clearly part of a Jordan-Chevalley decomposition of $\phi$ over $U$, we must have $\phi_s^U = \phi_s$ restricted to $U$. This proves the last statement of the proposition.
\end{proof}

    \begin{remark}
    If we naively take the semisimple part over each fiber, the resulting object need not even be meromorphic. Indeed, take the endomorphism $$f=\begin{pmatrix} 0 & z \\ 1 & 0 \end{pmatrix}$$ of the trivial rank $2$ bundle over $\C$. The critical set is just the point $0$. Note $\phi=\phi_s$ is generically semisimple, but not semisimple everywhere. 
    \end{remark}

\begin{remark}
    The components $\phi_s$ and $\phi_n$ can have poles. Consider for example the endomorphism 
    $$f = 
    \begin{pmatrix}
    0 & 1 & 0 \\ 
    0 & 0 & 1 \\
    0 & 0 & z \end{pmatrix}
    $$
    defined on the trivial bundle over $\C$. Its Jordan-Chevalley decomposition is
    $$ f_s + f_n =
    \begin{pmatrix}
    0 & 0 & 1/z \\ 
    0 & 0 & 1 \\
    0 & 0 & z \end{pmatrix}
    +
    \begin{pmatrix}
    0 & 1 & -1/z \\ 
    0 & 0 & 0 \\
    0 & 0 & 0 \end{pmatrix}.
    $$
\end{remark}
\begin{remark}
    The Higgs bundles $(E,\overline{\partial}_E,\phi_s)$ and $(E,\overline{\partial}_E,\phi_n)$ are not parabolic or wild Higgs bundles in the sense of \cite{S} or \cite{BBo}, since the spectral data is bounded.
\end{remark}

\subsection{The toral bundle}
Let $(E, \dbar_E, \phi)$ be a Higgs bundle on a surface $S$, and let $\phi_s$ be the semisimple part of $\phi$ over $S-B$. Since the center of the centralizer of an endomorphism $f$ is unchanged if we scale $f$, it makes sense to define $ZC(\phi_s)$ at each point of $S-B$.

\begin{defn}
    The toral bundle $F_\phi$ of the Higgs bundle $(E,\dbar_E,\phi)$ is the sub-bundle of $\mathrm{End}(E)$ over $S-B$ with fibers $ZC(\phi_s)$.
\end{defn}

The toral bundle is locally spanned by the projections $\pi_i$ of Proposition \ref{pispan}, and inherits the real structure $\dagger$ of Definition \ref{defn: real toral} and complex metric. We denote the corresponding real sub-bundle by $\R F_\phi$ and call it the real toral bundle. We emphasize that the real structure and metric on the toral bundle are independent of a choice of hermitian metric $h$ on $E$. On the other hand, given $h$ we get a map $\mathrm{Re}_h: F_\phi \to End_h(E)$ sending $s \mapsto \frac{1}{2}(s + s^{*_h}).$ $F_\phi$ admits a natural flat connection and flat metric under which $\pi_i$ are flat. The derivative of the section $s$ is given by
\[
ds = \partial s + \dbar s= \sum_i ds_i \pi_i.
\]

Given an ordering $E = \oplus_{i=1}^m E_i$ of the eigenbundles of $\phi_s$, we can apply the Schur decomposition to $\phi$ to write $\phi = \phi_a + \phi_u = \phi_a + (\phi_s)_u + \phi_n$. In general one can only find such a global ordering after pulling back to the cameral cover (see section \ref{sec: cameral covers}). Nonetheless, even if such a decomposition is not possible, we emphasize that $|\phi_a|^2$ and $|\phi_u|^2$ are globally defined smooth $(1,1)$-forms. If $\phi$ in $S_n(d,A)$, the Euclidean norm of $\phi_a$ is bounded above by $\sqrt{n}Ad.$

\subsection{Mochizuki's work and consequences}
Take note of the result below, which is contained in many sources (for example, \cite{Mo}) and proved using estimates of Simpson. We use it in the proof of Proposition \ref{asymdecoupling} below, and at a number of other points in the paper.
\begin{prop}[Proposition 2.1 from \cite{Mo}]\label{firstbound} Let $\phi$ be a Higgs field on $D(1)$ and $h$ a harmonic metric for $\phi$. On $D(r)$, there exists $C=C(n,r)$ such that $$|\phi|_{h}\leq C\max_{z\in D(1)}|\phi_{a}(z)|_{h}+C.$$
\end{prop} 
A nice feature of the proposition above is that, unlike Theorem \ref{GRSdecoupling} or Theorem B, the estimate does not depend on distance to the critical set. 

Recall Definition \ref{dfn:S(d,A)}: a Higgs field $\phi = fdz$ on the trivial holomorphic bundle over $D(1)$ is in $S_n(d,A)$ if the gap between distinct eigenvalues of $f$ is everywhere at least $d$ and all eigenvalues are bounded by $Ad$. 

One of the main ingredients toward Theorem B is Mochizuki's asymptotic orthogonality theorem (Theorem \ref{GRSdecoupling} above).
We deduce from this the following proposition.
 \begin{prop}\label{asymdecoupling}
     Let $\phi \in S_n(d,A)$. Write $\phi = f dz$, with Jordan-Chevalley decomposition $f = f_s + f_n$. Fixing $r<1,$ on $D(r)$, there are constants $C=C(n,r,A)>0$ and $c=c(n,r,A)>0$ such that
     \begin{equation}\label{commutatorbounds}
         |[f_s,f_s^{*_h}]|_{h}, \hspace{1mm} |[f_{n},f_s^{*_h}]|_{h}, \hspace{1mm}|[f_s,f_n^{*_h}]|_{h} \leq Ce^{-cd}
     \end{equation}
and 
\begin{equation} \label{commutatorbounds 2}
    |F(h)+[f_n,f_n^{*_h}]|_{h}\leq Ce^{-cd}.
\end{equation}
 \end{prop}
\begin{proof}
Fix an ordering of the generalized eigenspaces, and let $f = f_a + ({f_s})_u + f_n$ be the Jordan-Chevalley-Schur decomposition relative to this ordering. By Lemma \ref{lem: star and dagger} and Theorem \ref{GRSdecoupling}, we have $|f_s^* - f_s^\dagger|_h \leq \sqrt{2n}Ce^{-cd}|f_a|_h$. Since $\phi \in S_n(d,A)$, we have $|f_a|_h \leq \sqrt{n}dA$; hence there is a new constant $C = C(n, r, A)$ such that $|f_s^* - f_s^\dagger|_h \leq C e^{-cd}$.

On the other hand, by the definition of $f_s^\dagger$ (section \ref{DCthm}) we have $[f_s, f_s^\dagger] = 0$, and from the commutativity condition of the Jordan-Chevalley decomposition we also have $[f_n, f_s^\dagger] = 0$. Together with the previous bound this gives
\[
|[f_s,f_s^*]|_h = |[f_s,f_s^* - f_s^\dagger]|_h \leq Ce^{-cd}|f_s|_h
\]
and similarly
\begin{equation} \label{fn and fss}
|[f_n,f_s^*]|_h \leq Ce^{-cd}|f_n|_h.
\end{equation}
Next, we show that we have bounds $|f_s|_h \leq C d$ and $|f_n|_h \leq C (d+1)$ for some constant $C$, so that we can absorb them into the exponential decay at the cost of redefining $C$ as before. For the first one,
by Lemma \ref{lem: star and dagger}, we have
\[
|f_s|_h^2 = |f_a|_h^2 + |(f_s)_u|_h^2 \leq (1 + C^2e^{-2cd})|f_a|_h^2 \leq (1+C^2)(nA^2d^2).
\]

The linear-in-$d$ bound on $|f_n|_h$ follows from this together with the linear-in-$d$ bound on $|f|_h$ in Proposition \ref{firstbound} by the triangle inequality (see the remark below).

The third inequality on the first line is immediate as $[f_s, f_n^*]$ is just the adjoint of $[f_n, f_s^*]$. Finally the inequality $|F(h)+[f_n,f_n^{*_h}]|_{h}\leq Ce^{-cd}$ follows by plugging in the Jordan-Chevalley decomposition of $f$ into Hitchin's equation $F(h) + [f, f^*] = 0$ and applying the first three inequalities.
\end{proof}

\begin{remark}
    We could have avoided using Proposition \ref{firstbound} in the proof of \eqref{commutatorbounds 2} by using the stronger constant-in-$d$ bound in Theorem $B$ instead. Indeed, the proof of Theorem $B$ only uses the weaker equation \eqref{fn and fss}, so the reasoning is not circular.
\end{remark}

\section{Boundedness of the nilpotent part}\label{ThmAsection} 
In this section we prove Theorem B, giving an interior bound for the norm of the nilpotent part of a Higgs field $\phi$ in $S_n(d,A)$ in terms of $A$.
Throughout the section, let $(E,\overline{\partial}_E)$ be the trivial rank $n$ bundle over $D(1)$ and $(E,\overline{\partial}_E,\phi,h)$ a harmonic bundle, with $\phi \in S_n(d,A),$ and fix $r < 1$. Let $r_1 = (1+r)/2$, so we have $D(r)\subset D(r_1) \subset D(1).$ Write $\phi = f dz$, with Jordan-Chevalley decomposition $f = f_s + f_n$.

Our proof of Theorem B has the same starting point as the proof of the asymptotic orthogonality theorem in \cite{Mo} and also, in the case $s = \phi$, of Simpson's paper \cite[page 729]{S}, namely the following lemma.

\begin{lem} \cite[Lemma 2.4]{Mo} \label{lem: origsim}
Let $(E,\phi,h)$ be a harmonic bundle on the disk, with $\phi = fdz$. If $s$ is a section of $\mathrm{End}(E)$ with $[s,f] = 0 $, then
\begin{equation}\label{origsim}
    \partial_z\partial_{\overline{z}}\log |s|_{h}^2\geq \frac{|[s,f^{*_h}]|_h^2}{|s|_h^2}.
\end{equation}
\end{lem}

\begin{remark}
When $s$ is non-vanishing, (\ref{origsim}) is a consequence of the fact that curvature is decreasing in holomorphic sub-bundles.
\end{remark}

\begin{thm}[Theorem B]
 Let $\phi \in S_n(d,A)$ with $\phi=fdz$ and let $h$ be a harmonic metric for $\phi$. On any subdisk $D(r)\subset D(1)$, there exists $C=C(n,r,A)>0$ such that
$$|\phi_n|_h\leq C.$$
\end{thm}

\begin{proof}
We begin by applying Lemma \ref{lem: origsim} to the nilpotent part $f_n$ of the Higgs field, with the aim of showing that it is a subsolution of an elliptic equation.

\begin{align*}
        \partial_z \partial_{\overline{z}} \log |f_n|^2 &\geq \frac {|[f_n,f^*]|^2}{|f_n|^2} \\
        &= \frac{|[f_n, f_n^*] + [f_n, f_s^*]|^2}{|f_n|^2} \\
        &\geq \frac{|[f_n, f_n^*]|^2}{2|f_n|^2} - \frac{|[f_n, f_s^*]|^2}{|f_n|^2}.
\end{align*}
By Proposition \ref{SMprop}, the first term is bounded below by $C_0^{-2} |f_n|^2$ for some constant $C_0$ depending only on $n$. For the second term, we apply equation \eqref{fn and fss} in the proof of Proposition \ref{asymdecoupling}. The result is that on $D(r_1)$, 
\[
\partial_z \partial_{\overline{z}} \log |f_n|^2 \geq C_0^{-2}|f_n|^2 - C_1^2e^{-2cd}. 
\]

Finally we conclude by Proposition \ref{maxprinc} below that $|f_n| \leq \frac{C(1 + C_2e^{-2cd})}{(1-(r/r_1)^2)},$
which is clearly bounded above independently of $d$.
\end{proof}

\begin{prop} \label{maxprinc}
 Let $u:\overline{D(0,r)}\to [0,\infty)$ be a $C^2$ function that when non-zero satisfies 
    \begin{equation}\label{assumeddifineq}
        r^2\partial_z\partial_{\overline{z}}\log u \geq C^2_1u^2-C^2_2.
    \end{equation}
Then $u(z) \leq \frac{C_2 + 1}{C_1} \frac{r^2}{r^2 - |z|^2}$.
\end{prop}

\begin{proof}
Let $v(z) = \frac{C_2 + 1}{C_1} \frac{r^2}{r^2 - |z|^2}$. A simple calculation shows that $v$ is a supersolution to \eqref{assumeddifineq}:
\begin{align*}
    r^2\partial_z\partial_{\overline{z}}\log v &= \frac{r^4}{(r^2 - |z|^2)^2} \\
    &\leq C_1^2 v^2 - C_2^2.
\end{align*}
Since $v$ tends to infinity on the boundary of $D(0,r)$, it is greater than or equal to any subsolution that is bounded on the closed disk, which is exactly what we wanted to show. We remark that although the equation degenerates when $u=0$, the inequality $u \leq v$ holds trivially at these points so the maximum principle argument still works.
\end{proof}

One important consequence of Theorem B is the uniform curvature bound.
\begin{cor}\label{localbound}
    There exists $C=C(n,r,A)$ independent of $d$ such that $|F(h)|_h\leq C.$
\end{cor}
\begin{proof}
    By Proposition \ref{asymdecoupling}, there exists $C=C(n,r,A),c=c(n,r,A)>0$ such that in $D(r)$, $$|F(h)+[\phi_n,\phi_n^{*_h}]|_{h}\leq Ce^{-cR}.$$ Using Theorem B, we have $|\phi_n|_{h}\leq C(n,r,A)$. We thus find that $$|F(h)|_{h}\leq C|\phi_n|_{ h}^2 + Ce^{-cR}\leq C.$$ 
\end{proof}

\section{Asymptotics of flat connections}
We now prove Theorem C and Corollary C about the asymptotics of flat connections. Throughout the section, we continue our assumptions from section \ref{ThmAsection}. That is, $(E,\overline{\partial}_E)$ is the trivial rank $n$ bundle over $D(0,r)$ and $(E,\overline{\partial}_E,\phi,h)$ is a harmonic bundle, with $\phi \in S_n(d,A).$ We have generalized  eigenspace decomposition $\oplus_{i=1}^m E_i$ for $\phi$ with projections $\pi_i:E\to E_i$, and the usual Jordan-Chevalley decomposition. As well, we fix a subdisk $D(r)\subset D(1).$ To simplify notation, we'll use $\overline{\partial}$ for $\overline{\partial}_E,$ and $\partial_h$ for $\nabla_h^{1,0},$ acting on sections of $\textrm{End}(E)$ and all associated bundles.  As in section \ref{ThmAsection}, set $r_1=\frac{1+r}{2}.$ 

Working on a subdisk $D(r)\subset D(1),$ for each $p\in (1,\infty)$ and $\alpha\in (0,1)$, let $|\cdot|_{h,L^p(D(r))}$ and  $|\cdot|_{h,C^{0,\alpha}(D(r))}$ be the $L^p,$ and $C^{0,\alpha}$-norms respectively on sections of $E|_{D(r)}$ and associated bundles, defined with respect to the metric $h$. Using the standard Euclidean connection on $E$, for each $k\in\mathbb{Z}_+$ and $p\in (1,\infty),$ we similarly set $|\cdot|_{h,W^{k,p}(D(r))}$ to be the $W^{k,p}$ Sobolev norm on the relevant spaces of sections. For $L^p, C^{0,\alpha},$ and $W^{k,p}$ norms of forms with values in a bundle, we use the same notation, implicitly using the flat metric to measure the norms of form components.

\subsection{Holomorphic radial frames}
When working in any frame, we will use $|\cdot|$ for the Euclidean norm with respect to the frame, and likewise $|\cdot|_{L^p(D(r))}$, $|\cdot|_{W^{k,p}(D(r))}$, etc., for $L^p$ spaces and such, dropping the $``h"$ in the subcript to distinguish our notation. The $W^{k,p}$-space is defined with respect to the Euclidean metric in the frame. 

For an arbitrary Hermitian holomorphic vector bundle, control on the curvature of the Chern connection implies that in a particularly nice choice of frame, the connection form and other quantities are uniformly bounded.
\begin{prop}\label{prop:boundinframes}
    Let $(E,h_0)$ be a Hermitian holomorphic vector bundle over $D(1)$ with Chern connection $\nabla$. For every $C>0,$ there exists $C'(n,C,r)>0$ such that if $|F(\nabla)|_{h_0}\leq C,$ then we can find a local holomorphic frame in $D(r)$ in which the norms of the connection forms, the matrix of $h_0$, and the inverse matrix are all bounded above by $C'.$
\end{prop}
\begin{proof}
We first construct a ``radial frame," by which we mean a frame obtained by fixing a disk around a point with polar coordinates $(r,\theta)$, and parallel transporting that frame at the center along the rays $\{\theta =\textrm{constant}\}.$ If $\Omega$ is the connection form, then by construction, $\Omega(\partial_r)=0,$ and $\Omega(\partial_\theta)$ is controlled by the curvature (see \cite[section 2.2.1]{DK}).

 Following the procedure from the proof of the Koszul-Malgrange theorem, as explained in \cite[section 2.2.1]{DK}, after passing to $D(r)$ one can find a gauge transformation that transforms our radial frame to a holomorphic frame, and in which the norm of $\Omega$ is uniformly controlled in terms of the norm in the previous frame. 
 
 Set $H_0$ to be the matrix representing $h_0$. Since $\Omega=H_0^{-1}\partial H_0,$ $H_0$ is determined by its entries at the center and parallel transport using $\Omega$, and hence $H_0$ is uniformly bounded.  The same reasoning applies to $H_0^{-1}.$ 
\end{proof}

\begin{defn}
   We refer to a frame constructed in the fashion above as a holomorphic radial frame.
\end{defn}

Making use of a holomorphic radial frame, we can prove that the operator $\overline{\partial}\partial_h$ admits interior elliptic estimates that do not depend on $d$, which we'll use in the proof of Proposition \ref{dels} below. 
\begin{lem}\label{lem: schauder estimate}
    For all $1<p<\infty$ there exists $C=C(n,r,p,A)$ such that for all sections $u$ of $\textrm{End}(E)$,
    $$|u|_{h,W^{2,p}(D(r))}\leq C(|\overline{\partial}\partial_h u|_{h,L^p(D(r))}+| u|_{h,L^p(D(1))}).$$
\end{lem}
\begin{proof}
The quantities of the estimate above do not depend on the frame, so we are welcome to work in a holomorphic radial frame. Let $H$ be the matrix of $h$ in this frame, so that $H^{-1}\partial_z H dz$ is the connection form.
From the formula $\partial_h u = \partial_z u dz + [H^{-1}\partial_z H dz,u],$
\begin{equation}\label{eqn: laplacian for h}
    \overline{\partial}\partial_h u = (\partial_{\overline{z}}\partial_z u + [H^{-1}\partial_z H ,\partial_{\overline{z}}u]+[\partial_{\overline{z}}(H^{-1}\partial_z H ),u])d\overline{z}\wedge dz.
\end{equation}
Note that $F(h) = \partial_{\overline{z}}(H^{-1}\partial_z H)d\overline{z}\wedge dz,$ and hence by Corollary \ref{localbound} and Proposition \ref{prop:boundinframes}, the coefficients of $\overline{\partial}\partial_h$ are uniformly bounded, depending on $n,r,$ and $A$, but not on $d$.
 Thus, in the Euclidean norm we have the interior elliptic estimate
 $$|u|_{W^{2,p}(D(r))}\leq C(n,r,p,A)(|\overline{\partial}\partial_h u|_{L^p(D(1))}+| u|_{L^p(D(1))})$$ (see \cite[Theorem 9.11]{GT} for the estimate in the case of functions and \cite[Theorem 10.33]{Ni} for treating sections of bundles). Since $H$ and $H^{-1}$ are uniformly bounded, we obtain a comparable estimate when we change Euclidean norms to $h$-norms.
\end{proof}

\subsection{Almost parallelity}
The main step toward Theorem C is the proposition below.
\begin{prop}\label{dels}
    For all $\alpha\in [0,1),$ there exists $C=C(n,r,\alpha,A)$, $c=c(n,r,\alpha,A)>0$ such that for every $i,$
         $$|\partial_h \pi_i|_{h,C^{0,\alpha}(D(r))}\leq Ce^{-cd}.$$
\end{prop}
When $\phi$ is generically regular semisimple, the result above is Proposition 2.10 in \cite{Mo}. In the general case, we have to bound some extra terms.
\begin{proof}[Proof of Proposition \ref{dels}]
First observe that for any holomorphic section $s$ of $\textrm{End}(E)$, $$\overline{\partial}\partial_h s =[s,[\phi,\phi^{*_h}]].$$ Indeed, $$\overline{\partial}\partial_h s=\overline{\partial}\partial_hs - \partial_h\overline{\partial}s=[F(h),s]= [s,[\phi,\phi^{*_h}]],$$ where in the last equality we used the Hitchin equation.

    Now, for notational convenience, set $s=\pi_i$. We first show that there exists $C=C(n,r_1,A)$ and $c=c(n,r_1,A)$ such that on $D(r),$ 
    \begin{equation}\label{eq: deldelbars}
        |\overline{\partial}\partial_h s|_h\leq Ce^{-cd}.
    \end{equation}
 Since $[s,\phi]=0,$ the Jacobi identity yields that $$\overline{\partial}\partial_h s= [s,[\phi,\phi^{*_h}]]=[\phi,[s,\phi^{*_h}]].$$ Taking norms and applying Proposition \ref{firstbound}, we have $$|\overline{\partial}\partial_h s|_{h}\leq \sqrt{2}|\phi|_{h}|[s,\phi^{*_h}]|_{h}\leq Cd|[s,\phi^{*_h}]|_h.$$
     Note that $s=s^\dagger$ and $$|[s^\dagger,\phi^{*_h}]|_h= |[s^\dagger-s^{*_h},\phi^{*_h}]|_h\leq 2|\phi|_h|s^\dagger-s^{*_h}|_h.$$ By Proposition \ref{firstbound} and Theorem \ref{GRSdecoupling} combined with Lemma \ref{lem: star and dagger}, we deduce $$|[s,\phi^{*_h}]|_h=|[s^\dagger,\phi^{*_h}]|_h\leq AdCe^{-cd}|s|_h.$$ 
From part (1) of Lemma \ref{lem:almostorthogonal1}, $|s|_h$ is uniformly bounded, and the claim follows after enlarging the constants $C$ and $c$ in order to absorb $Ad$.

To deduce the proposition, by holomorphicity of $s$, $\partial_h s^{*_h}=0$, and hence (\ref{eq: deldelbars}) yields $$|\overline{\partial}\partial_h (s-s^{*_h})|\leq Ce^{-cd}.$$ By Lemma \ref{lem: schauder estimate}, and because $L^p$ norms on a disk are controlled by $L^\infty$ norms, for all $1<p<\infty$ we have $C=C(n,r,p,A)>0$ such that $$|s-s^{*_h}|_{h,W^{2,p}(D(r))}\leq C(|\overline{\partial}\partial_h (s-s^{*_h})|_{h}+|s-s^{*_h}|_{h}),$$ where the $C^0$-norms above are taken over $D(r_1).$ Hence, by Theorem \ref{GRSdecoupling} again and the estimate above, $$|s-s^{*_h}|_{h,W^{2,p}(D(r))}\leq Ce^{-cd}.$$ Via Morrey's theorem, for all $\alpha\in (0,1)$ we can find $C=C(n,r,\alpha,A)>0$ such that $$|s-s^{*_h}|_{h,C^{1,\alpha}(D(r))}\leq C|s-s^{*_h}|_{h,W^{2,p}(D(r))}\leq Ce^{-cd},$$ from which we obtain $$|\partial_h s|_{h,C^{0,\alpha}(D(r))}=|\partial_h (s-s^{*_h})|_{h,C^{0,\alpha}(D(r))}\leq Ce^{-cd},$$ as desired. 
    \end{proof}

We recall that we say $E_i$ is said to be $\epsilon$-almost parallel if $|B_i|_h \leq \epsilon$, where $B_i$ is the second fundamental form of $E_i$ in $E$, which is a one-form valued in the bundle $\mathrm{Hom}(E_i, E_i^\perp)$.

\begin{thm}[Theorem C]
    Let $(E,\overline{\partial}_E,\phi,h)$ be a harmonic bundle on $D(1)$ with $E=D(1)\times \C^n$ and $\phi\in S_n(d,A).$
    On any subdisk $D(r)\subset D(1)$, there exist constants $C=C(n,r ,A), c=c(n,r,A)>0$ such that each sub-bundle $E_i$ is $Ce^{-cd}$-almost parallel for $h$.
\end{thm}

\begin{proof}
As in the proof of Proposition 2.11 in \cite{Mo}, this is an algebraic consequence of the $C^1$ bound in Proposition \ref{dels}. Since $\overline{\partial} \pi_i = 0$, Proposition \ref{dels} shows that $|\nabla_h \pi_i| \leq C e^{-cd}$. Now let $\pi_i'$ be the $h$-orthogonal projection to $E_i$. With respect to the splitting $E = E_i \oplus E_i^\perp$ we have
\[
\nabla_h \pi'_i = \begin{pmatrix}
    0 & B_i \\
    B_i^{*_h} & 0 \\
\end{pmatrix},
\]
where $B_i$ is the second fundamental form of $E_i$. On the other hand, if $s_i \in E_i$, then $(\nabla_h \pi_i)s_i$ and $(\nabla_h \pi_i')s_i$ are congruent modulo $E_i$, so the upper right part of $\nabla_h \pi_i$ is also equal to $B_i$. Therefore 
\[
\frac{1}{2}|\nabla_h \pi'_i|^2_h =  |B_i|^2_h \leq  |\nabla_h \pi_i|^2_h \leq Ce^{-cd}.
\]
This shows that $E_i$ is almost parallel.
\end{proof}

From the Gauss equation $F_i = F + B_i^* \wedge B_i$, it follows: 

\begin{cor} \label{cor: curvature of sub-bundles} In the setting of Theorem C, the difference between the curvature $F_i$ of the sub-bundle $E_i$ and the restriction to $E_i$ of the curvature $F$ of $E$ is bounded by $C^2e^{-2cd}.$ 
\end{cor}

\subsection{Asymptotic expansion of the flat connection}
Here we prove Corollary C. Let's recall the setup from the introduction. We are given the flat connection $D_h=\nabla_h+\phi+\phi^{*_h}$, and we would like to find an approximation for when $d$ is large.

Relative to our generalized eigenspace decomposition $E=\oplus_{i=1}^m E_i$, let $\phi_i$ be the eigen-$1$-forms of $\phi_i,$ i.e., $\phi_s=\sum_{i=1}^m \phi_i\pi_i.$ As well, let $\phi_n$ be the nilpotent part. Set $h_i$ to be the restriction of $h$ to $E_i$ and $h^{\oplus}=\oplus_{i=1}^m h_i.$ We denote the Chern connection of $h^{\oplus}$ by $\nabla_h^{\oplus},$ which splits as $$\nabla_h^{\oplus}= \sum_{i=1}^m \partial_i + \overline{\partial}_i,$$ where $\overline{\partial}_i=\overline{\partial}_E|_{E_i}$  and $\partial_i$ is the $(1,0)$ component of the Chern connection for $h_i.$ The statement of Corollary C is that on $D(r),$
\begin{equation}\label{asymptotic2}
    D_h = \nabla_{h}^{\oplus}+\sum_{i=1}^m(\phi_i+\overline{\phi}_i)\pi_i+ \phi_n +\phi_n^{*_{h^{\oplus}}}+a,
\end{equation}
where $a$ is an $\textrm{End}(E)$-valued $1$-form such that $|a|_{h}\leq Ce^{-cd},$ for constants
$C=C(n,r,A), c=c(n,r,A)>0$.

As in section \ref{ThmAsection}, let $\pi_i'$ be the 
the $h$-orthogonal projection onto $E_i$. By uniqueness of the Chern connection, for any section $s$ of $\textrm{End}(E_i)$, $\overline{\partial}_i s=\overline{\partial}_E s$ and $\partial_i s=\pi_{i}'\circ \partial_h s.$ 
    \begin{proof}[Proof of Corollary C]
    Decomposing $D_h$ as $$D_h = \nabla_h + \phi_s+\phi_s^{*_h}+\phi_n+\phi_n^{*_h},$$ we see that proving (\ref{asymptotic2}) amounts to showing that on $D(r),$ $$ |\phi_s+\phi_s^{*_h}-\sum_{i=1}^m (\phi_i\pi_i+\overline{\phi}_i\pi_i) |_{h}, |\phi_n^{*_h}-\phi_n^{*_{h^{\oplus}}}|_h, \textrm{ and } |\partial_h - \sum_{i=1}^m \partial_i|_h$$ are all bounded above by $Ce^{-cd}.$ For the first item, note that $\phi_s^\dagger=\Big (\sum_{i=1}^m \phi_i\pi_i\Big )^\dagger= \sum_{i=1}^m \overline{\phi_i}\pi_i,$ so by Lemma \ref{lem: star and dagger} and Theorem \ref{GRSdecoupling}, on $D(r)$ and for constants $C,c$,
$$\Big |\phi_s+\phi_s^{*_h}-\sum_{i=1}^m (\phi_i\pi_i+\overline{\phi}_i\pi_i)\Big |_{h}\leq Ce^{-cd}.$$ 

For the $\phi_n$ term, recall that if $H$ is the matrix-valued function determined by $h,$ then $\phi_n^{*_h}=H^{-1}\overline{\phi}_n^\vee H, $ where the $\vee$ is the transpose operator, and likewise for $h^{\oplus}$. Since the splitting $E=\oplus_{i=1}^m E_i$ is $h^\oplus$-orthogonal, using this expression, the inequality $|\phi_n^{*_h}-\phi_n^{*_{h^{\oplus}}}|_h\leq Ce^{-cd}$ follows from Theorem \ref{GRSdecoupling} and Theorem B.

Proving $|\partial_h - \sum_{i=1}^m \partial_i|_h\leq Ce^{-cd}$ amounts to finding the right estimate on the action of $\partial_h-\partial_i$ on sections of $E_i$.
For any section $s$ of $E_i$, by the Leibniz rule, $$(\partial_h-\partial_i) s  = \pi_i'\circ \pi_i \circ \partial_h s - \pi_i'\partial_h(\pi_i s)=\pi_i'\circ \partial_h\pi_i(s).$$ Therefore, by Proposition \ref{dels}, $$|\partial_h s - \partial_i s|_{h}= |\pi_i'\circ \partial_h\pi_i(s)|_{h}\leq |\partial_h\pi_i(s)|_{h}\leq Ce^{-cd}|s|_h.$$ The result follows. 
    \end{proof}

\subsection{Higher order estimates}
We prove higher order estimates that will come into play in the proof of Theorem A.  The procedure is a bootstrap and we follow \cite[section 2.25]{Mo}, although the results are naturally weaker than that of \cite{Mo}.

We use the proposition below to choose a good local frame for $E$.  
\begin{prop}\label{prop:E_iframe}
    On any subdisk $D(r)\subset D(1)$, there exists $C=C(n,r, A)>0$ such that, for each $i,$ we can find a holomorphic frame for $E_i$ in $D(r)$ in which the norms of the connection forms, the matrix of $h|_{E_i},$ and the inverse matrix are all bounded above by $C$.
\end{prop}
\begin{proof}
   By Corollary \ref{localbound}, $F(h)$ is uniformly bounded depending only on $n,r,$ and $A$. By Corollary \ref{cor: curvature of sub-bundles}, the same bound holds for $F(h|_{E_i}),$ with a slight change in the constant. We then apply Proposition \ref{prop:boundinframes}.
\end{proof}
Putting frames from Proposition \ref{prop:E_iframe} together, we obtain a frame for all of $E$, which we'll call an adapted holomorphic radial frame. Let $H$ and $\phi=\Phi dz$ be the matrix representations of $h$ and $\phi$ respectively in this frame. By design, $\Phi = \sum_{i=1}^m \Phi_i,$ where $\Phi_i$ is an endomorphism of $E_i|_{D(r_1)}.$ We write matrix entries of $H$ as $H_{ij}^{\alpha\beta}=h(u_i^\alpha,u_j^\beta).$ 
\begin{lem}\label{localhobounds}
There exists $C=C(n,r,A), c=c(n,r,A)>0$ such that in $D(r)$, for $i\neq j$, 
 $|\partial_z H_{ij}^{\alpha\beta}|,|\partial_{\overline{z}}H_{ij}^{\alpha\beta}|, |\partial_{\overline{z}}\partial_z H_{ij}^{\alpha\beta}|\leq Ce^{-cd},$ and for $i=j$ these quantities are all bounded above by $C$.
\end{lem}
\begin{proof}
Since we have uniform bounds for the connection form $H^{-1}\partial H$, the curvature $\overline{\partial}(H^{-1}\partial H)$, and $H$ itself, and $H$ is real, all entries of $\partial_z H$, $\partial_{\overline{z}}H$, and $\partial_{\overline{z}}\partial_z H$ are uniformly bounded.

We now assume $i\neq j$. Since $H$ is real, the norms of $\partial_z H$ and $\partial_{\overline{z}} H$ agree, so we only need to find a bound on $\partial_z H$. Let $u_i$ be any element of the holomorphic radial frame for $E_i$, and $u_j$ in the frame for $E_j$. Our assumption on the frame implies that $u_i$ and $\partial_h u_i$ are uniformly bounded in $h$-norm, and likewise for $u_j$.  
By holomorphicity, $$\partial_z h(u_i,u_j)=h((\partial_h \pi_i)u_i,u_j) + h(\pi_i (\partial_h u_i), u_j),$$ and hence by Theorem \ref{GRSdecoupling} and Proposition \ref{dels}, 
   $$|h((\partial_h \pi_i)u_i,u_j)|\leq |\partial_h \pi_i|_h |u_i|_h|u_j|_h\leq Ce^{-cd}.$$
The proof of Corollary C shows that $|\partial_h u_i - \partial_i u_i|\leq Ce^{-cd}|u_i|_h.$ Hence, for the second term, by Theorem \ref{GRSdecoupling} again, $$|h(\pi_i (\partial_h u_i), u_j)|\leq |h(\partial_i u_i, u_j)|+Ce^{-cd}\leq Ce^{-cd}|\partial_i u_i|_h |u_j|_h +Ce^{-cd} \leq Ce^{-cd}.$$ This establishes the first result. For the second order term, Proposition \ref{asymdecoupling} implies $$F(h)=[\phi_n,\phi_n^{*_h}] +a,$$ where $a$ is an endomorphism whose $h$-norm is exponentially small. Observe that $[\phi_n,\phi_n^{*_{h_\infty}}]$ is a sum of endomorphisms of the $E_i$'s. We previously showed that the norm of $\phi_n^{*_h}$ differs from that of $\phi_n^{*_{h_\infty}}$ by an exponential error, and it follows that the $\textrm{Hom}(E_i,E_j)$ components of the matrix representing $[\phi_n,\phi_n^{*_h}]$ have the same exponential decay. 
\end{proof}
We say that, in our frame, a matrix is nearly block if it's uniformly bounded and if for $i\neq j$, the $\textrm{Hom}(E_i,E_j)$ components are exponentially small in $d$. We now promote to asymptotic decoupling results for any order, showing that every derivative of $H$ is nearly block. 
\begin{prop}\label{holocalbounds}
    In an adapted holomorphic radial frame, for every pair of non-negative integers $\ell_1,\ell_2,$ there exist constants $C=C(n,r,\ell_1,\ell_2,A),$ $c=c(n,r,\ell_1,\ell_2,A)$ such that on $D(r),$ for $i\neq j,$ $ |\partial_z^{\ell_1}\partial_{\overline{z}}^{\ell_2} H_{ij}^{\alpha\beta}|\leq Ce^{-cd}$, and for $i=j,$ $|\partial_z^{\ell_1}\partial_{\overline{z}}^{\ell_2} H_{ij}^{\alpha\beta}|\leq C$. 
\end{prop}

 \begin{proof}
 We'll repeatedly use the basic fact that the product of two nearly block matrices is nearly block. As well, we use throughout that the Cauchy estimates give linear bounds on the derivatives of $\Phi$ in terms of $d$. 
 
 The known estimate on $\partial_{\overline{z}}\partial_z H_{ij}^{\alpha\beta}$, applied in the disk $D(r_1)$, where we remind that $r_1=\frac{1+r}{2}$, together with elliptic regularity, give the desired bounds for the cases $(\ell_1,\ell_2)=(2,0)$ or $(0,2)$ on $D(r)$. We'll now bootstrap on the Hitchin equation, which is written in our frame as
\begin{equation}\label{equation for H}
    H^{-1}\partial_{\overline{z}}\partial_z H - H^{-1}\partial_{\overline{z}}H\cdot H^{-1}\partial_z H-[\Phi,H^{-1}\overline{\Phi}^\vee H]=0.
\end{equation}
We go by induction on $k=\ell_1+\ell_2,$ with the cases $k=1$ and $k=2$ already established. Assume the result holds for $k-1$. We first treat a pair $(\ell_1,\ell_2)$ with $k=\ell_1+\ell_2$ and $\ell_1,\ell_2>0$. Since $r<1$ is arbitrary, we can assume that we have the bounds for the case $k-1$ in $D(r_1)$ rather than $D(r)$. Applying $\partial_z^{\ell_1-1}\partial_{\overline{z}}^{\ell_2-1}$ to the equation (\ref{equation for H}), we get $$H^{-1}\partial_z^{\ell_1}\partial_{\overline{z}}^{\ell_2}H -[\partial_z^{\ell_1-1}\Phi,H^{-1}\overline{\partial_{z}^{\ell_2-1}\Phi^\vee}H]+G=0,$$ where $G$ is a linear combination of matrices that are products of derivatives of $H$, $\Phi$, $\overline{\Phi^\vee}$, and derivatives of $\Phi$ and $\overline{\Phi^\vee}$. By the induction hypothesis and using the $(2,0)$ and $(0,2)$ cases, $G$ is uniformly bounded by $C$ and the $\textrm{Hom}(E_i,E_j)$ components of $G$ are bounded above by $Ce^{-cd}$, with $C$ and $c$ depending on $n,r,\ell_1,$ $\ell_2$, and $A$. The same is true for the commutator term, by the previous reasoning from the proof of Lemma \ref{localhobounds} and the Cauchy bound.  Thus, $\partial_z^{\ell_1}\partial_{\overline{z}}^{\ell_2}H$ is $H$ times a matrix that's bounded and has the correct decay in the $\textrm{Hom}(E_i,E_j)$ components. By our choice of frames, multiplying by $H$ preserves this structure, and so we've verified the result for such $(\ell_1,\ell_2).$ 

We're left to handle the pairs $(k,0)$ and $(0,k).$ The estimates follow easily from elliptic regularity for the Euclidean Laplacian: for $(k,0)$, $\partial_z^k H_{ij}^{\alpha\beta}= \partial_z^{2}(\partial_z^{k-2} H_{ij}^{\alpha\beta})$ is controlled by the sum of the norms of $\partial_{\overline{z}}\partial_z(\partial_z^{k-2} H_{ij}^{\alpha\beta})= \partial_{\overline{z}}\partial_z^{k-1} H_{ij}^{\alpha\beta}$ and $\partial_z^{k-2} H_{ij}^{\alpha\beta}$, which at this point we know satisfy the required inequalities. The same argument applies for $(0,k)$. Putting it all together, we've completed the induction step.
\end{proof}

\section{Asymptotic decoupling}
In this section, we use the estimates of the previous sections to prove Theorem A.

\begin{thm}[Theorem A]
Let $S$ be an arbitrary Riemann surface and  $(E,\overline{\partial}_E,\phi)$ a Higgs bundle on $S$ with critical set $B$. Let $R_i$ be sequence in $\R^+$ tending to infinity, and let $h_i$, be any sequence of harmonic metrics on $(E,\overline{\partial}_E,R_i\phi)$. 
Passing to a subsequence, there exist $s_i \in \mathcal{G}^\oplus$ such that the triples $(s_i^*\overline{\partial}_E,  s_i^* R_i \phi_n, s_i^* h_{R_i})$ converge in $C^\infty_{\mathrm{loc}}(S - B)$ to $(\pi_*\dbar_\infty, \pi_*\psi_\infty, \pi_* h_\infty)$ for some harmonic bundle structure $(\dbar_\infty, \psi_\infty, h_\infty)$ on $\hat{E}$ with $\psi_\infty$ nilpotent. 
\end{thm}

We divide the proof into two parts. First, we show in Proposition \ref{prop: bundle convergence} that there is a harmonic bundle $(E_\infty, \psi_\infty, h_\infty)$ to which the data $(E, \dbar_E, R\phi_n, h_R)$ sub-converge with respect to fixed local charts. We then use Uhlenbeck's patching argument to show that one can take a sequence of smooth bundle isomorphisms $\sigma_R: E_\infty \to E$ such that this convergence can be viewed as convergence of the triple $(\sigma_R^*\dbar_E, \sigma_R^* (R \phi_n), \sigma_R^* h_R)$ on the bundle $E_\infty$. We recover Theorem A by taking the gauge transformations $s_R = \sigma_R\circ \mu$ for some fixed isomorphism $\mu: E \to E_\infty$.

\subsection{Constructing the limiting bundle}

In Theorem A, we require that our gauge transformations $s_R$ lie in the subgroup $\mathcal{G}^{\oplus}$ of the gauge group that preserves the decomposition of $E$ into eigenspaces of $\phi_s$. To arrange this, we will make sure $E_\infty$ has a similar decomposition and all maps preserve this decomposition. In other words, we will construct $E_\infty$ as the pushforward of a bundle $\hat{E}_\infty$ over the spectral curve $\hat{S}$ of $S - B$, and ask that the maps $\sigma_R$ and $\mu$ are induced from isomorphisms of $\hat{E}$ and $\hat{E}_\infty$. Recall that $\hat{E}$ is the bundle over the spectral curve whose pushforward is $E$. Similarly, if $U = \bigsqcup U_\alpha$ is a covering of $S - B$ by open sets, we let $\hat{U} = U \times_{S - B} \hat{S}$ be the induced covering of the spectral curve, and if $\hat{\tau}$ is an isomorphism of bundles over the spectral curve, then $\tau$ is the corresponding isomorphism between their pushforwards to $S-B$.

\begin{prop} \label{prop: bundle convergence} Let $S$ be an arbitrary Riemann surface and  $(E,\overline{\partial}_E,\phi)$ a Higgs bundle on $S$ with critical set $B$. Let $\mathcal{R}$ be a proper sequence in $\R^+$ and let $h_R$, $R \in \mathcal{R}$, be any sequence of harmonic metrics on $(E,\overline{\partial}_E,R\phi)$. One can choose an open covering $U \to S-B$, a subsequence $\mathcal{R}'$, and a sequence $\hat{\tau}_R$ ($R \in \mathcal{R}'$) of holomorphic trivializations of $\hat{E}$ over $\hat{U}$, such that $\tau_R^*h_R$ and $\tau_R^*R\phi_n$ converge in $C^\infty_\mathrm{loc}(U)$, and the gluing maps $(\tau_R)_1^{-1} \circ (\tau_R)_2$ converge in $C^\infty_\mathrm{loc}(U \times_S U)$.
\end{prop}

\begin{proof}
For each point $p \in S - B$, choose a local coordinate at $p$ such that $D(1)$ is compactly contained in $S-B$, and let $d$ and $A$ be such that $\phi \in S_n(d,A)$ on $D(1)$. Let $U = \sqcup_\alpha U_\alpha$ be a locally finite covering of $S-B$ by open sets $U_\alpha \cong D(1/2)$ in these coordinates. From Corollary \ref{localbound}, we see that the curvature $F(h_R)$ of the Chern connection $\nabla_R$ is locally bounded on $S$ independently of $R$. Let $\hat{h}_R$ be the metric on $\hat{E}$ induced from $h_R$ by restriction to each generalized eigenspace. From Corollary \ref{cor: curvature of sub-bundles} it follows that the curvature of the Chern connection of $\hat{h}_R$ is also locally bounded on $S$. Let $p\mapsto C(p)$ be a continuous function on $S$ dominating this curvature. 

Since $U_\alpha$ is contractible, each $\hat{U}_\alpha$ is a union of copies of $U_\alpha$. By Proposition \ref{prop:E_iframe}, we can choose a frame $\hat{\tau}_R$ of $\hat{E}$ over each $\hat{U}_\alpha$ such that if $\hat{H}_R := \hat{\tau}_R^*\hat{h}_R$ is the hermitian metric in this frame, then $\hat{H}_R$ is the identity matrix at each $\hat{p}_\alpha$ and the connection matrix $\hat{H}_R^{-1} \partial \hat{H}_R$ is bounded in terms of $C(p)$. By integration, we get control on $\hat{H}_R$ on all of $\hat{U}_\alpha$. Let $H_R = \tau_R^* h_R$ be the entire metric in this frame (so that $\hat{H}_R$ is just the block-diagonal entries of $H_R$). Then by asymptotic orthogonality, Theorem \ref{GRSdecoupling}, and the bound on $\hat{H}_R$, we see that $H_R$ is also controlled at $p_\alpha$ by $C(p)$. It furthermore follows from Proposition \ref{holocalbounds} that each derivative of the matrix $H_R$ is also bounded in terms of $C(p)$. Hence by the Arzel{\`a}-Ascoli theorem, there is a smoothly convergent subsequence of the $H_R$, converging to a metric $H_\infty$. Using again the asymptotic orthogonality theorem as $R$ tends to infinity, it follows that $H_\infty$ is block-diagonal, i.e., that it is the pushforward of a metric $\hat{H}_\infty$ on each $\hat{U}_\alpha$.

We next show that upon passing to a further subsequence we can arrange that $R\Phi_n := \tau_R^*R\phi_n$, the pullback of the nilpotent part of the rescaled Higgs field, also locally converges. By the corollary to Theorem B, the norm $|R\phi_n|_{h_R}$ is bounded in terms of $C(p)$. Since $|R\phi_n|_{h_R} = |R\Phi_n|_{H_R}$ and $H_R$ are controlled, the Euclidean norm $|R\Phi_n|$ is also controlled by $C(p)$. Since both $\tau_R$ and $R\Phi_n$ are moreover holomorphic, $R\Phi_n$ form a normal family, so there is a locally convergent subsequence, with limit $\Psi_\infty$.

Now let $\pi_1$ and $\pi_2$ be the two projections from $\hat{U} \times_{\hat{S}} \hat{U}$ to $\hat{U}$, and let $(\hat{\tau}_R)_j = \pi_j^* \hat{\tau}_R$, $j=1,2$. Set $\hat{g}_R = (\hat{\tau}_R)_1^{-1} \circ (\hat{\tau}_R)_2$. We may view $\hat{g}_R$ as a matrix-valued function on $\hat{U} \times_{\hat{S}} \hat{U}$. Since the trivializations $\hat{\tau}_R$ are holomorphic, so are the $\hat{g}_R$. Since the matrices $H_R$ are locally bounded, the $\hat{g}_R$ land in a compact subset of $GL({n_i}, \C)$ and therefore form a normal family. Hence, we can extract a sub-sequential $C^{\infty}_{\mathrm{loc}}$ limit $\hat{g}_\infty$.
\end{proof}

With $\hat{g}_R$ as above, the cocycle condition $\hat{g}_{12}\hat{g}_{23} = \hat{g}_{13}$ on $\hat{U} \times_{\hat{S}} \hat{U} \times_{\hat{S}} \hat{U}$ continues to hold in the limit, so $\hat{g}_\infty$ defines a holomorphic bundle $\hat{E}_\infty$ on $\hat{S}$ by descent. Similarly, $H_\infty$ and $\Phi_\infty$ continue to obey the corresponding cocycle conditions on $U \times_{S-B} U$, so they determine a metric $h_\infty$ and a Higgs field $\psi_\infty$ on $E_\infty$. 

Since the convergence is smooth, one can take limits in the estimates of Proposition \ref{asymdecoupling} to see that the metric $h_\infty$ is a solution to the decoupled equations
$$F(h_\infty) + [\psi_\infty,\psi_\infty^{*_{h_\infty}}]=0, \hspace{1mm} [\phi_s,\phi_s^{*_{h_{\infty}}}] =[\phi_s,\psi_\infty^{*_{h_\infty}}]=0,$$
in which the latter two merely express that both $h_\infty$ and $\psi_\infty$ are descended from their hatted versions on the spectral curve.

\subsection{Limits of vector bundles and proof of Theorem A}

Having constructed the tuple $(E_\infty, \dbar_\infty , h_\infty, \psi_\infty)$ solving the decoupled Hitchin equations \eqref{decoupled equations} as a limit of $(E, \dbar, h_R, R\phi_n)$ in the sense of Proposition \ref{prop: bundle convergence}, it is tempting to stop and say we are done. Indeed, a reasonable notion of a sequential limit of such tuples over $S-B$ is such a tuple over the smooth space $(S - B) \times (\mathcal{R} \cup \{\infty\})$ (with smooth structure on the latter given by its embedding into $\mathbb{RP}^1$) together with an isomorphism between its restriction to $(S-B) \times \mathcal{R}$ and our original sequence. This is exactly what we construct in Proposition \ref{prop: bundle convergence}, the choice of isomorphism being $\tau_R$.

In fact, such limits are even unique, in the sense that if $(E, \dbar, h_R, \phi_R)_j$, $j = 1,2$, are two tuples over $(S - B) \times (\mathcal{R} \cup \{\infty\})$, and $\tau$ is an isomorphism between their restrictions to $(S-B) \times \mathcal{R}$, then there exists an isomorphism between the restrictions to $(S-B) \times \{\infty\}$. This is proved by the same logic as Proposition \ref{prop: bundle convergence}: the sequence $\tau_R$ of isomorphisms, viewed as sections of $\mathrm{Hom}(E_1, E_2)$, are holomorphic and bounded, so they have a convergent subsequence \cite[Prop 2.3.15]{DK}. This Hausdorffness statement shows that the limiting tuple $(E_\infty, \dbar_\infty , h_\infty, \psi_\infty)$ depends a priori only on the subsequence of $\mathcal{R}$ and not on the local trivializations. The upshot is that it is not really necessary to transport this convergence back to the bundle $E$ as we do in Theorem A, but we will nonetheless. 

\begin{proof}[Proof of Theorem A]
Following the strategy outlined at the beginning of this section, our aim is to construct a sequence of smooth bundle isomorphisms $\hat{\sigma}_R: \hat{E}_{\infty} \to \hat{E}$ such that the triple $(\hat{\sigma}_R^* \dbar_E, \hat{\sigma}_R^*(R\phi_n), \hat{\sigma}_R^*h_R)$ converges smoothly locally on $E_\infty$ to $(\dbar_\infty, \psi_\infty, h_\infty)$. Since $E$ and $E_\infty$ are both presented in terms of their trivializations on $U$, a bundle isomorphism between them is equivalently a refinement $U'$ of the covering $U$, and smooth matrix-valued functions $\hat{\sigma}_R$ on $U'$ such that $\hat{g}_\infty = (\hat{\sigma}_R)_1^{-1}\hat{g}_R (\hat{\sigma}_R)_2$ on $\hat{U}' \times_{\hat{S}} \hat{U}'$. As long as $\hat{\sigma}_R$ converges locally smoothly to the identity as $R$ tends to infinity, the smooth convergence of $H_R$ and $R\Phi_n$ to $H_\infty$ and $\Psi_\infty$ will imply smooth convergence on $E_\infty$.

Uhlenbeck's patching argument \cite[Proposition 3.2]{Uh1} establishes the existence of such smooth $\hat{\sigma}_R$ when the base manifold and the gauge group are both compact. The proof mirrors the vanishing of $H^1$ for fine sheaves of abelian groups, but in the nonabelian setting.

To apply this argument, we first address the non-compactness of the gauge group. For each $R$, let $\hat{\rho}_R = \sqrt{{\hat{H}_R}^{-1}\hat{H}_\infty}$ as a matrix-valued function on $\hat{U}$. Then $\hat{\rho}_R$ converges smoothly to the identity, and pulls back the metric $\hat{H}_R$ to $\hat{H}_\infty$. The remaining gauge freedom is in the unitary group of $\hat{H}_\infty$, a compact group.

Second, we address the (usual) case where $\hat{S}$ is noncompact. Fix an exhaustion by compact subsurfaces $K_i$ with boundary. On each compact subsurface, Uhlenbeck's argument produces a sequence of gauge transformations $\hat{\sigma}_{i,R}$ satisfying the requirements. Since the complement of $\overline{K}_i$ is an open surface and the gauge group is connected, the bundles $\hat{E}$ and $\hat{E}_\infty$ are both trivializable on the complement of $K_i$. Similarly, $\hat{\sigma}_{R,i}$ can be extended in some way to a smooth bundle isomorphism on all of $\hat{S}$. Of course, since our aim is $C^\infty_{loc}$ convergence, the choice of extension is not important. Then the desired sequence of gauge transformations is of the form $\hat{\sigma}_{R, i(R)}$, for any function $i(R)$ tending to infinity.
\end{proof}

\section{Harmonic maps to symmetric spaces and buildings}\label{hars}
Throughout the section, let $S$ be a Riemann surface with universal cover $\tilde{S}$. 

\subsection{Lie groups and Lie algebras}\label{Liegroupssection}
 Let $G$ be a connected and semisimple Lie group with no compact factors and with Lie algebra $\g$. Let $\nu$ be an $\textrm{Ad}(G)$-invariant bilinear form on $\g$ that is negative definite on Lie subalgebras of compact subgroups (for example, one could take the product of the Killing forms of the simple factors). Given a maximal compact subgroup $K\subset G,$ we always write the Lie algebra as $\kk\subset \g$, and the $\nu$-orthogonal complement as $\p$. Note that $\p$ does not depend on the choice of $\nu$.

Let $\textrm{ad}:\g\to \textrm{End}(\g)$ be the adjoint representation. An element $X\in\g$ is semisimple if $\textrm{ad}(X)\in \textrm{End}(\g)$ is a semisimple endomorphism, and nilpotent if $\textrm{ad}(X)$ is nilpotent. The rank of $G$ is the dimension of a maximal abelian subalgebra of $\g$ contained in $\p$. We refer to such an algebra as a maximal split toral subalgebra of $\p$ (we reserve the word Cartan for the category of complex Lie algebras). The word toral is meant to emphasize that it necessarily consists of semisimple elements \cite[Lemma 8.1]{Hu}.

Fix a maximal split toral subalgebra $\ft$ of $\p$. The Weyl group of $\ft$ (sometimes called the restricted Weyl group of $\ft$), denoted $W,$ is the normalizer of $\ft$ in $K$ modulo its centralizer. Let $K^{\C}\subset G^{\C}$ and $\ft^{\C}\subset \p^{\C}\subset \g^{\C}$ be the complexifications of $K$, $\ft$, and $\p$ respectively. Let $\mathcal{O}(\p^{\C})^{K^{\C}}$ be the ring of $K^{\C}$-invariant complex polynomials on $\p^{\C}$, and $\mathcal{O}(\ft^{\C})^W$ the ring of $W$-invariant polynomials on $\ft^{\C}$. The real version of the Chevalley restriction theorem says that the restriction map from $\mathcal{O}(\p^{\C})^{K^{\C}} \to \mathcal{O}(\ft^{\C})^W$ is an isomorphism \cite[Theorem 6.10]{Helg}, \cite[Theorem 7]{Vin}. Another theorem attributed to Chevalley asserts that $\mathcal{O}(\ft^{\C})^W$ is free on $r$ generators, where $r$ is the dimension of $\ft$, i.e., the rank of $G$ \cite[page 54]{Hum}. Thus, $\mathcal{O}(\p^{\C})^{K^{\C}}$ is free on $r$ generators.

\subsection{Symmetric spaces and buildings}\label{sec: symmetric spaces}
A symmetric space of non-compact type is a connected and simply connected Riemannian manifold $(N,\nu)$ with an inversion symmetry about each point, and whose de Rham decomposition contains only non-compact symmetric spaces and no factors of $\R$ (see \cite[Chapter V, \S 4]{Helg}). The isometry group of $(N,\nu)$ is always semi-simple with no compact factors and trivial centre.

A point $x$ in $N$ canonically determines a Lie algebra $\g$ of Killing fields and an orthogonal decomposition $\g = \kk_x \oplus\p_x$, with $\kk_x$ the Killing fields vanishing at $x$. Since $\nu$ is invariant under $G=\textrm{Aut}_0(N)$, $\nu_x$ determines $\nu$. The subgroup $K = \mathrm{Stab}_G(x)$ is maximally compact, with Lie algebra $\kk_x$, but $N$ does not uniquely determine the Lie group $G$: from any cover of $G$, the quotient by the maximal compact subgroup of this cover recovers $N.$ Via the Maurer-Cartan form of $G$, the metric $\nu$ on $N$ is equivalent to a $K$-invariant positive definite bilinear form on $\p_x$. The rank of the symmetric space $N$ is the largest dimension of a subspace of its tangent space at any point on which the sectional curvature vanishes, or equivalently the dimension of any maximal split toral subalgebra of $\p_x$. Under this definition, the rank is equal to the rank of the corresponding Lie group.

As we'll explain in section 8, affine buildings arise as asymptotic cones of symmetric spaces, and harmonic maps to symmetric spaces converge in a suitable sense to harmonic maps to buildings. Since we're not going to delve into the geometry of buildings so deeply, we won't give full definitions. For definitions and references, we refer the reader to \cite{Cu}, \cite[section 2]{KNPS}, and \cite[section 8]{LTW}. Very informally, an affine Euclidean building $(\mathcal{B},d)$ modelled on a maximal split toral subalgebra $\ft$ of $\p\subset \g$ is a non-positively curved (NPC) length space comprised of flat subspaces called apartments, which are copies of $\ft$. The flats are glued along hyperplanes, called walls, in a way that respects particular affine Weyl group actions on the apartments. Note that an affine building need not be locally compact. An affine chart $\ft\to A\subset \mathcal{B}$ is an identification from $\ft$ to an apartment. An atlas $\{f:\ft\to \mathcal{B}\}_{f\in A}$ for $\mathcal{B}$ is a cover by affine charts that is compatible with the structure of the building.

\subsection{Harmonic maps}\label{hmapsdefinitions}
Since we will work with harmonic maps to both symmetric spaces and buildings, we define harmonic maps in the metric space context. We fix a metric $\sigma$ on $S$ that is compatible with the complex structure of $S$, and use $\sigma$ as well to denote the lift to $\tilde{S}$. We use the conventions of Korevaar-Schoen \cite{KS}. Let $(M,d)$ be a complete and NPC length space and let $h:\tilde{S}\to (M,d)$ be a locally Lipschitz map. Korevaar-Schoen \cite[Theorem 2.3.2]{KS} associate a locally $L^1$ measurable metric $g=g(h)$, which plays the role of the pullback metric. If $M$ is a smooth manifold, and the distance $d$ is induced by a Riemannian metric $\delta$, then $g(h)$ is represented almost everywhere by the pullback metric $h^*\delta$. The energy density is the locally $L^1$ function 
\begin{equation}\label{mease}
    e(h)=\frac{1}{2}\textrm{trace}_\sigma g(h).
\end{equation}
On any relatively compact subsurface $U\subset S,$ the total energy is
\begin{equation}\label{tot}
    \mathcal{E}(U,h) = \int_U e(h)dA,
\end{equation}
where $dA$ is the area form of $\sigma$. We comment here that the measurable $2$-form $e(h)dA$ does not depend on the choice of compatible metric $\sigma$, but only on the complex structure. 

\begin{defn}\label{def: harmonic}
$h$ is harmonic if for all relatively compact subsurfaces $U\subset S,$ $h$ is a minimizer for $f\mapsto \mathcal{E}(U,f)$ among the Lipschitz maps on $\overline{U}$ that agree with $h$ on $\partial U.$
\end{defn}
\begin{remark}
    In \cite{KS}, the authors define Sobolev spaces $W^{1,2}(\cdot, M)$ for maps to $(M,d)$, on which one can define total energies of functions. They define $h$ to be harmonic if for all relatively compact $U,$ it minimizes the total energy among $W^{1,2}$ maps with the same trace (see \cite[\S 1.12]{KS}). By Theorem 2.2 and Theorem 2.4.6 in \cite{KS}, harmonic maps are locally Lipschitz, and it suffices to minimize among locally Lipschitz maps.
\end{remark}

In the Riemannian setting, $h$ is harmonic if and only if it satisfies a second order semilinear elliptic equation: we view the derivative $dh$ as a section of $f^*TM\otimes T^*S$. Let $\partial h$ be the $(1,0)$-part of the derivative of $h$, and $\nabla^{0,1}$ the $(0,1)$-part of the connection on $h^*TM \otimes_\R \C$ induced from the Levi-Civita connection of $M$, as well as its natural extension to $h^*TM\otimes_\R \C$-valued forms. Then $h$ is harmonic if and only if
\begin{equation}\label{hmapequation}
    \nabla^{0,1}\partial h = 0.
\end{equation}

\subsection{Harmonic maps to symmetric spaces and $G$-Higgs bundles}
Let $(G,K)$ be as in section \ref{Liegroupssection}, and let $\g^{\C} = \kk^{\C} \oplus\p^{\C}$ be the $K^{\C}$-invariant decomposition of the Lie algebra. If $P$ is a holomorphic principal $K^{\C}$-bundle on $S$, we write $\textrm{ad}\p^{\C} = P \times_{K^{\C}} \p^{\C}$.  

\begin{defn}\label{GHiggs}
A $G$-Higgs bundle is a pair $(P,\phi),$ where $P$ is a holomorphic principal $K^{\C}$-bundle on $S$ and $\phi$ is a holomorphic section of $\textrm{ad}\p^{\C}\otimes \mathcal{K}$ called the Higgs field. 
\end{defn}

Hitchin's equations for $G$-Higgs bundles are as follows. A reduction of the structure group $P^h\subset P$ to $K \subset K^{\C}$ defines a real structure $*_h$ as well as a Chern connection $A_h$ on $P^h$. The curvature $F(h)$ of the Chern connection is a $2$-form valued in the adjoint bundle of $P^h$.
\begin{defn}
    We say that $P^h$ solves Hitchin's self-duality equations if 
    \begin{equation}\label{SD}
        F(h) + [\phi,\phi^{*_h}]=0.
    \end{equation}
    We call $(P,\phi,*_h)$ a harmonic $G$-bundle.
\end{defn}
The reduction $P^h$ gives a solution to (\ref{SD}) if and only if  $(P^h \times_K G,A_h + \phi + \phi^{*_h})$ is a flat principal $G$-bundle.

A $\rho$-equivariant harmonic map $h$ from $\tilde{S}$ to a symmetric space $(N,\nu)$ is equivalent to a harmonic $G$-bundle. We now explain this equivalence. Here we're allowing $\pi_1(S)$ to be trivial. Let $N_\rho\to S$ be the flat bundle associated to the action of $\pi_1(S)$ on $\tilde{S}$ and $\rho$ on $N$. From $\nu$, $N_\rho$ inherits a metric on the fibers. Let $Q\to S$ be the principal $G$-bundle whose fiber at $z \in S$ is the $G$-torsor of isometries from $N$ to the fiber of $N_\rho$ at $z$, which inherits a flat connection $A$ (an equivariant $\g$-valued $1$-form satisfying certain properties) from the flat connection on $N_\rho$. If we fix a point $x$ in $N$ with stabilizer $K$ and corresponding decomposition $\g = \kk\oplus \p$, then the map $h$ yields a section of $N_\rho$, which we'll still denote by $h,$ and thereby gives a reduction of the structure group of  $Q^h\subset Q$ to $K$, whose fiber at $z$ is the subset of isometries sending $x$ to $h(z)$. Using the Maurer-Cartan form, the pullback of the vertical tangent bundle of $N_\rho$ by the section $h$ is identified with $Q^h \times_K \p$. Write $\nabla_h$ for the pullback of the Levi-Civita connection by $h$ to $Q^h \times_K \p$. Let $P$ be the complexification of $Q^h$, now a $K^{\C}$-bundle. Complexifying the tensor bundle to $P \times_{K^{\C}} \p^{\C},$ the $(0,1)$-part of $\nabla_h$ determines a complex structure $\overline{\partial}$. If we denote by $\phi$ the $(1,0)$-part of the pullback of the Maurer-Cartan form under $h,$ then the equation (\ref{hmapequation}) exactly says that it is holomorphic with respect to our complex structure, i.e., $$\overline{\partial}\phi=0.$$
Putting everything together, we have a harmonic $G$-bundle in $(P,\phi,*_h)$. On the other hand, if we start from $(P,\phi)$ and pick a point in $N$, the reduction of structure group $P^h$ is equivalent to a map from the universal cover to $(N,\nu),$ which is necessarily equivariant with respect to the holonomy of the connection $A_h + \phi+\phi^{*_h}.$ The equation $\overline{\partial}\phi=0$ is equivalent to the harmonicity of this map.

A faithful representation $\sigma:G\to \textrm{GL}(n,\C)$ turns a $G$-Higgs bundle into an ordinary Higgs bundle. Indeed, the associated vector bundle $E=\tilde{S}\times_{\sigma} \C^n$ inherits a complex structure from $P$, and $d\sigma:\g\to \mathfrak{gl}(n,\C)$ induces a map from the sheaf of sections of $\textrm{ad}\p^{\C}$ to that of $\textrm{End}(E)$. The Higgs field $\phi$ for the $G$-Higgs bundle becomes a Higgs field $\phi_\sigma$ on $E$, and a solution to (\ref{SD}) determines a solution to (\ref{1}) (see also \cite[section 2.4]{SS}). The analysis of harmonic maps to symmetric spaces then reduces to that of Hitchin's equations for ordinary Higgs bundles. We often compute in the adjoint representation $G\to \textrm{Aut}(\g)$, which takes a $G$-Higgs bundle to a Higgs bundle with underlying bundle modeled on $\g^{\C}$.  

Finally, we point out that the notions introduced for Higgs bundles in section \ref{prelimhiggs} have abstract formulations. 
To prove the analogue of Proposition \ref{JCdec}, apply the argument of Proposition \ref{JCdec} in the adjoint representation. Since the Lie group is semisimple, the adjoint representation is injective and we obtain a decomposition $\phi=\phi_s+\phi_n$, with $\phi_s$ and $\phi_n$ meromorphic sections of $\textrm{ad}\p^{\C}$. The proof of Proposition \ref{JCdec} shows that $\phi_s$ and $\phi_n$ are holomorphic on the complement of the critical set for the ordinary Higgs bundle associated with the adjoint representation.

The critical set can be defined without reference to a representation; we gave such a definition in \cite[section 3.4]{SS} for generically semisimple Higgs bundles, and the construction can be easily adapted to the general case. Since we work in representations to apply our local analytic results, we won't need such an abstract definition. However, it must be remarked that if we start with a $G$-Higgs bundle with critical set $B_0$, and then choose a representation to obtain an ordinary Higgs bundle with critical set $B$, then $B$ might be larger than $B_0$. 

\subsection{The Hitchin map}\label{HBsection}
We continue with $(G,K)$ as above with Lie algebra splitting $\g=\kk\oplus \p$. As well, fix  a maximal split toral subalgebra $\ft\subset \p$ with Weyl group $W$. Recall that the affine quotient $\ft^{\C}//W$ is the spectrum of the ring of $W$-invariant polynomials on $\ft^{\C}$.  

Let $\mathcal{K} \otimes \ft^{\C}$ be the tensor product of the canonical bundle $\mathcal{K}$ with the trivial holomorphic $\ft^{\C}$-bundle over $S$. We now consider the quotient by the Weyl group action on the total space of $\mathcal{K}\otimes\ft^{\C}$ coming from the trivial action on $\mathcal{K}$ and the standard action on the vector space $\ft^{\C}$. Since $\mathcal{K}\otimes\ft^{\C}$ is not affine, it does not quite make sense to take the affine quotient. However, it is affine when restricted to any affine subset of $S$, so we can make sense of the quotient $\mathcal{K}\otimes \ft^{\C}//W$ by gluing together the affine quotients on a covering of $S$.

The result is a fiber bundle over $S$ whose fiber is noncanonically isomorphic to $\ft^{\C}//W$, the isomorphism depending on a trivialization of the cotangent bundle at that point. We call a section of the fiber bundle $\mathcal{K}\otimes \ft^{\C}//W$ a $\ft^{\C}//W$-valued $1$-form.

\begin{defn}
    The Hitchin base for $G$-Higgs bundles over $S$ is the space of holomorphic $\ft^{\C}//W$-valued $1$-forms on $S$.
\end{defn}

This definition can be made more concrete by choosing homogeneous generators $p_1, \ldots, p_l$ of the graded ring of $W$-invariant polynomials on $\ft^{\C}$ with degrees $m_1, \ldots, m_l$. Applying $(p_1,\dots, p_l)$ to a $\ft^{\C}$-valued $1$-form on $S$ gives a point in $\oplus_{i = 1}^l H^0(S, \mathcal{K}^{m_i})$. This identifies the Hitchin base with $\oplus_{i = 1}^l H^0(S, \mathcal{K}^{m_i})$. We remark that the natural $\C^*$ action on the Hitchin base coming from the $\C^*$ action on $\mathcal{K}$ induces the action $\lambda \cdot p_i = \lambda^{m_i}p_i$ in this representation.

We now define the Hitchin map. The inclusion $\ft^{\C} \to \p^{\C}$ induces a natural map $\ft^{\C}//W \to \p^{\C}//K^{\C}$, which is an isomorphism by the Chevalley theorem. We call the inverse map the Chevalley map. Let $(P,\phi)$ be a $G$-Higgs bundle, so that $\phi$ is a holomorphic section of $\mathrm{ad}\p^{\C}$. There is a natural map from $\mathrm{ad}\p^{\C}$ to $S \times \ft^{\C}//W$, as follows.
Since $\mathrm{ad}\p^{\C} = P \times \p^{\C} // K^{\C}$ with $K^{\C}$ acting by $(\cdot k^{-1}, k \cdot)$, if we quotient again by $K^{\C}$ acting only on the $P$ factor, the result is the same as $P/K^{\C} \times \p^{\C}//K^{\C}$. The first factor is just $S$, and by the Chevalley theorem, the second factor is $\ft^{\C}//W$.

Since the $K^{\C}$ actions all commute with scaling, the Chevalley map induces a further map $\cdot//K^{\C}: \mathcal{K} \otimes \mathrm{Ad}\mathfrak{p}^{\C} \to \mathcal{K} \otimes \ft^{\C}//W$. 

\begin{defn}
    The Hitchin map is the map from the set of $G$-Higgs bundles on $S$ to the Hitchin base sending a $G$-Higgs bundle $(P,\phi)$ to the point $\phi//K^{\C}$ in the Hitchin base.
\end{defn}

\subsection{Cameral covers}\label{sec: cameral covers}

There is an equivalent description of the Hitchin base in terms of cameral covers, which we now describe. The scheme-theoretic cameral cover associated to a $\ft^{\C}//W$-valued $1$-form $\eta$ is the pullback of the quotient map $\mathcal{K} \otimes \ft^{\C} \to \mathcal{K} \otimes \ft^{\C} // W$ by $\eta$.
    \begin{center}
    \begin{tikzcd}
        \mathcal{K} \otimes \ft^{\C} \arrow[r]
        & \mathcal{K} \otimes \ft^{\C} // W \\
        C \arrow[u, "\eta_C"] \arrow[r, "\pi"]
        \arrow[ur, phantom, "\scalebox{1.5}{$\urcorner$}" , very near start, color=black]
        & S \arrow[u,"\eta"]
    \end{tikzcd}
    \end{center}
Instead of this universal object, we work with what we call a small cameral cover in \cite[section 3.2]{SS}.
\begin{defn}
    A small cameral cover is a connected Riemann surface $C$ together with maps $\pi:C\to S$ and $\eta_C: C\to \mathcal{K}\otimes \ft^{\C}$ making the diagram above commute.
\end{defn}
While any two small cameral covers are isomorphic, they are not universal because the isomorphism is not unique.

A small cameral cover does not come with an action of $W$, but it is nonetheless a Galois branched covering of $S$ whose Galois group $W_C$ can be identified with a subquotient of $W$. When the Higgs field is generically regular semisimple, any connected component of the normalization of the scheme theoretic cameral cover is a small cameral cover, and $W_C$ is a subgroup of $W$.

We can also view $\eta_C$ as a holomorphic section of $\mathcal{K}_C \otimes \ft^{\C}$ using the natural map $\pi^*\mathcal{K}_S \to \mathcal{K}_C$, which is called a $\ft^{\C}$-valued $1$-form.

From the harmonic maps standpoint, $\ft^{\C}$-valued $1$-forms give a useful perspective on the Hitchin base. Given a $\ft^{\C}$-valued $1$-form $\eta_C$ on $C$, $\xi=\eta_C+\overline{\eta}_C$ defines a $\ft$-valued $1$-form (a section of $T^*C\otimes \ft$), whose components in any identification of $\ft$ with $\R^{\textrm{rank}(G)}$ are harmonic. Viewing $\ft$ as an abelian group under addition $+$, $\xi$ represents a cohomology class in $H^1(C,\ft)$. Choosing a basepoint $x$ on $C$, the cohomology class of $\xi$ determines a representation $\chi: \pi_1(C,x)\to (\ft,+).$ We can then find covering spaces $\tilde{C}$ of $C$ on which the lift integrates to give an equivariant harmonic map to $\ft$. Our main results on harmonic maps are saying that such equivariant harmonic maps determine the limiting intrinsic behaviour of high energy harmonic maps to symmetric space.

\subsection{Harmonic maps to buildings}\label{sec: harmonic to building}
A harmonic map $h$ to an affine building does not carry the full information of a Higgs bundle, but it does give rise to a $\ft^{\C}//W$-valued $1$-form $\eta$. When that $\ft^{\C}//W$-valued $1$-form is called $\eta$, the authors in \cite{KNPS} call the map $h$ a harmonic $\eta$-map.  In their construction in \cite[section 3]{KNPS}, the authors restrict to buildings modeled on a maximal split toral subalgebra of $\mathfrak{sl}(n,\C).$ We present here a slightly different but equivalent construction, and we work in the general setting of semisimple Lie algebras.

Let $\mathcal{B}$ be an affine building modeled on the maximal split toral subalgebra $\ft$ of $\p\subset \g$ and let $h:\tilde{S}\to (\mathcal{B},d)$ be an equivariant Lipschitz map. A point $z\in S$ is regular if for every lift $x$ of $z$ in $\tilde{S}$ there exists a neighbourhood $U$ of $x$ such that $h(U)$ is contained in a single apartment. We denote the set of regular points by $S_{reg}.$ The complement of $S_{reg}$ is called the singular set. 

Let $z\in S_{reg}$, and let $x$ be a lift of $z$ to $\tilde{S}$ with neighbourhood $U$ such that $h(U)$ is contained in a single apartment $A\subset \mathcal{B}$. Any affine chart $f: f^{-1}(V)\to V\subset A$ determines linear coordinate functions $x_1,\dots, x_l$ on $f^{-1}(V)\subset \ft$ whose zero loci are reflection hyperplanes that generate the action of $W$. Choosing the chart to contain $h(U)$, differentiating $h$ in such linear coordinates determines an $\ft$-valued $1$-form in $U$.

If $h$ is harmonic, $h|_U$ is represented by a harmonic function to a Euclidean space, and hence the $(1,0)$-component of this $\ft$-valued $1$-form is a holomorphic $\ft^{\C}$-valued $1$-form. If we use a different affine chart around $h(U)$, the $\ft^{\C}$-valued $1$-form transforms by the action of $W$. Hence, we can associate a $\ft^{\C}//W$-valued $1$-form $\eta$ on $U$, which depends only on $h$ and no choices. To extend $\eta$ to all of $\tilde{S}$, take note of the following result about the singular set, which is due to Gromov-Schoen for locally finite buildings \cite[Theorem 6.4]{GSb}, and Breiner-Dees-Mese in general \cite[Theorem 1.1]{BDM}. 
\begin{prop}\label{prop: singular set}
     The singular set of a harmonic map from a Lipschitz Riemannian domain into a Euclidean building has Hausdorff codimension at least $2$.
\end{prop}
Note we've only defined harmonic maps to metric spaces from Riemann surfaces, but Definition \ref{def: harmonic} can be made just the same for any Riemannian manifold source. 

If we choose a metric on $S$, then taking the pointwise norm of $\eta$ recovers the measurable energy density of $f$ with respect to this metric, which is locally uniformly bounded because $f$ is locally Lipschitz. It then follows from classical results in complex analysis that we can extend $\eta$ to all of $\tilde{S}$ (see, for example, \cite{Sha}). By equivariance, $\eta$ descends to a holomorphic $\ft^{\C}//W$-valued $1$-form on $S$.
\begin{remark}
    It is conjectured in \cite{KNPS} that for each $\ft^{\C}//W$-valued $1$-form, there is a building or building-like object with a harmonic map with this $\ft^{\C}//W$-valued $1$-form that is in some sense universal \cite[Conjecture 1.17]{KNPS} (see also \cite{KNPS2}).
\end{remark}

\section{High energy harmonic maps and convergence to buildings}
In this section we prove Theorem E and Corollary E. We first prove the comparison between pullback metrics, then the comparison between pullback connections, and then we discuss harmonic maps to buildings. 

\subsection{Pullback metrics at high energy}
Recalling the setup for Theorem E, $S$ is a Riemann surface, $G$ is a semisimple real Lie group of non-compact type with maximal compact subgroup $K$ and Lie algebra decomposition $\g=\kk\oplus \p$, and $(P,R\phi,*_{h_R})$, $R>0$, is a ray of harmonic $G$-bundles over $S$.  We fix as well a maximal split toral subalgebra $\ft\subset \p$. The symmetric space is $(G/K,\nu)$, and as explained in the introduction, $\nu$ induces a Euclidean metric $m$ on $\ft$. As well, we use $|\cdot|_{h_R}$ for the norm induced by $h_R$, and $|\cdot|$ for the norm induced by $m$. Let $f_R$ be the harmonic map to $(G/K,\nu)$, $\eta$ the $\ft^{\C}$-valued $1$-form on a small cameral cover $C$ obtained by using the Hitchin map on $(P,\phi)$, and $f_\eta:\tilde{C}\to (\ft,m)$ the harmonic map from the universal covering of $C$. We take the adjoint representation of $G$ in order to apply the estimates that we've found for Higgs bundles. We will abuse notation and write $\phi$ for $\textrm{ad}(\phi)$, $\phi_s$ for $\textrm{ad}(\phi_s)$, etc. Let $B$ be the critical set for this ordinary Higgs bundle. 

Now that we're working globally, and not just on the disk, to measure certain norms we fix a metric $\sigma$ on $S$. When taking the norm of a form with values in a bundle we'll add a sub-scripted $\sigma$ to our norm; for instance, we write $|\cdot|_{h_R,\sigma}$ and $|\cdot|_{\sigma}$. With no $\sigma$, we apply $|\cdot|_{h_R}$ and $|\cdot|$ to forms in the usual way. For example, under our convention here, $|\phi|_{h_R}^2$ is a metric and not a function.  

The first part of Theorem E is the statement that on any open subset $U$ of $S-B,$ $f_R^*\nu(z)$ is uniformly comparable to $R^2f_\eta^*m$. Toward this, let us first write the pullback metrics in terms of the Higgs bundle data. Using the Maurer-Cartan form,
\begin{align*}
    f_R^*\nu&=\nu(\partial f_R, \overline{\partial} f_R)+\nu( \partial f_R, \partial f_R) + \nu(\overline{\partial f_R}, \overline{ \partial f_R}) \\
&=R^2|\phi|_{h_R}^2+R^2\nu(\phi,\phi)+R^2\nu(\overline{\phi},\overline{\phi}),
\end{align*}
and 
\begin{align*}
    f_\eta^*\nu&=m(\partial f_\eta, \overline{\partial} f_\eta)+m( \partial f_\eta,  \partial f_\eta) + m(\overline{\partial f}_\eta, \overline{ \partial f}_\eta) \\
&=|\eta|^2+m(\eta,\eta)+m(\overline{\eta},\overline{\eta}).
\end{align*}
Note that, since both maps are equivariant with respect to actions by isometries, and the Deck group of $C$ over $S$, i.e., $W_C$, is acting by isometries on $(\ft,m)$, both pullback metrics descend all the way to $S$. Henceforth we view them both as symmetric tensors on $S$.

Over every point, $\phi$ is in the $\textrm{Ad}|_{K^{\C}}$ orbit of the $\ft^{\C}//W$ valued $1$-form on $S$. Hence $\nu(\phi,\phi)=m(\eta,\eta),$ and the first part of Theorem E will follow immediately from the proposition below. 
\begin{prop}\label{normcomparison}
    On an open subset $U\subset S-B$, there exist constants $C_1=C_1(G,U,\sigma, \eta), C_2=C_2(G,U,\sigma), c=(G,U,\sigma, \eta)>0$ such that
    $$R^2|\eta|_{\sigma}^2\leq R^2|\phi|_{h_R,\sigma}^2 \leq R^2|\eta|_{\sigma}^2 + C_1e^{-cR}+C_2.$$ 
\end{prop}
\begin{remark}
    By our formula for the pullback metrics, Proposition \ref{normcomparison} is equivalent to the statement that on $U$, $$R^2 e(f_\eta)\leq e(f_R)\leq R^2e(f_\eta) + C_1e^{-cR}+C_2,$$ with $C_j,c$ as above.
\end{remark}
The cleanest way to prove Proposition \ref{normcomparison} is to use the Schur decomposition. Recall that while $\phi=\phi_a+\phi_u$ only makes sense locally, $|\phi_a|_{h_R}^2$ and $|\phi_u|_{h_R}^2$ are globally well-defined. Note as well that $|\phi_a|_{h_R}=|\eta|$. 
\begin{proof}[Proof of Proposition \ref{normcomparison}]
We split up 
\begin{equation}\label{expanded norm}
|\phi|_{h_R}^2 =|\phi_{a}|_{h_R,\sigma}^2+|\phi_u|_{h_R,\sigma}^2 = |\eta|_{\sigma}^2+|(\phi_s)_u + \phi_n|_{h_R,\sigma}^2\leq |\eta|_\sigma^2+|(\phi_s)_u|_{h_R,\sigma}^2 + |\phi_n|_{h_R,\sigma}^2.
\end{equation}
By Lemma \ref{lem:almostorthogonal1} combined with Theorem \ref{GRSdecoupling}, for $z\in S-B$ we can find $C_1=C_1(G,U,\sigma,\eta),c=c(G,U,\sigma,\eta)>0$ such that $|(\phi_s)_u|_{h_R,\sigma}^2\leq C_1e^{-cR}$ (we did the same bound in the proof of Proposition \ref{asymdecoupling}). By Theorem B, for $z\in S-B$, there exists $C_2=C_2(G,U,\sigma)$ such that $|\phi_n|_{h_R,\sigma}\leq C_2.$ In all of these bounds, the only dependence on $G$ comes from the dimension of the underlying holomorphic vector bundle, which is $\dim \g^{\C}$. Inserting the bounds on $|(\phi_s)_u|_{h_R,\sigma}^2$ and $|\phi_n|_{h_R,\sigma}$ into (\ref{expanded norm}), we obtain Proposition \ref{normcomparison}.
\end{proof}

\subsection{Pullback connections at high energy}
Here we explain how to obtain the second part of Theorem E using Higgs bundles and Proposition \ref{dels}. We continue with the objects and notations from the previous subsection. 

The second part of Theorem E concerns a certain comparison between the two pullback connections $\nabla^{f_R^*}$ and $\nabla^{f_\eta^*},$ which live a priori on $f_R^*TG/K$ and $f_\eta^*T\mathfrak{t}$ respectively. To define this comparison, first recall that if $g$ is an endomorphism of a Hermitian vector space $(V,h),$ we have the real centre of the centralizer $\R ZC(g_s)$. Although $\R ZC(g_s)$ is naturally a sub-algebra of $\textrm{End}(V)$, it is not necessarily real with respect to the real structure $*_h$ on $\textrm{End}(V)$. To compare $*_h$ on $\textrm{End}(V)$  and $\dagger$ on $ZC(g),$ we define a map $\textrm{Re}_h: \R ZC(g_s)\to \textrm{End}(V)$, going into the $*_h$-self adjoint endomorphisms and given by $s\mapsto \frac{1}{2}(s+s^{*_h}).$ In the context of our family of harmonic bundles $(E,\overline{\partial}_E,R\phi,h_R)$ on $S$ obtained through the harmonic $G$-bundles $(P,R\phi,*_{h_R})$ and the adjoint representation (please excuse the double use of notation), for every $R$, patching the corresponding maps $\textrm{Re}_{h_R}$ on each fiber determines a map $\textrm{Re}_{h_R}: \R F_\phi\to \textrm{End}_{h_R}(E)$. 
Using the Maurer-Cartan form as in section \ref{hmapsdefinitions}, $f_R^*TG/K$ identifies with the adjoint bundle of the reduction of $P$ to $K$ solving the self-duality equations and $f_\eta^*T\mathfrak{t}$ identifies with the trivial $\mathfrak{t}$-bundle over $C$. 
The bundle $f_R^*TG/K$ descends to a bundle on $S$ that under the adjoint representation identifies with $\textrm{End}_{h_R}(E).$ As well,
using the Deck group action, the trivial $\mathfrak{t}$-bundle $f_\eta^*T\mathfrak{t}$ further descends to a bundle $\R F_\eta$ over $S-B$ that identifies with the real toral bundle $\R F_\phi$ of the adjoint representation of $(P,\phi).$  Thus, under our identifications, $\textrm{Hom}(\R F_\eta,f_R^*TG/K)$ becomes $\textrm{Hom}(\R F_\phi,\textrm{End}_{h_R}(E))$. The map $\textrm{Re}_R$ relating the two connections (and discussed in the introduction) is what we get by transporting $\textrm{Re}_{h_R}$ to $\textrm{Hom}(\R F_\eta,f_R^*TG/K)$.  Both $\textrm{Re}_R \circ \nabla^{f_\eta^*}$ and $\nabla^{f_R^*}\circ \textrm{Re}_R$ are $1$-forms valued in $\textrm{Hom}(\R F_\eta,f_R^*TG/K),$ and the bundle $\textrm{Hom}(\R F_\eta,f_R^*TG/K)$ has a norm $|\cdot|_{m^*\otimes f_R^*\nu}$ coming from tensoring the dual of $m$ and $f_R^*\nu.$

\begin{prop}\label{prop:thmE2ndpart}
    Let $U\subset S-B$ be an open subset. There exist constants $C=C(G, U, \sigma), c=c(G,U,\sigma)>0,$ such that

$$    |\textrm{Re}_R \circ \nabla^{f_\eta^*}-\nabla^{f_R^*}  \circ \textrm{Re}_R|_{m^*\otimes f_R^*\nu,\sigma}\leq Ce^{-cR}.$$
\end{prop}
\begin{proof}
To prevent the notation from becoming confusing, in this proof we differentiate our notation and use $\nu$ for the metric on the symmetric space, and $\nu_0$ for the associated bilinear form on $\g$ and the bilinear form on $\textrm{End}(E)$ induced from $\nu_0$ and the adjoint representation.

We work in $\textrm{Hom}(\R F_\phi,\textrm{End}_{h_R}(E))$, which inherits a norm $|\cdot|_{\nu_0^*\otimes h_R}$ from $\nu_0$ and $h_R$. Using the identifications above, $\nabla^{f_R^*}$ becomes $\nabla_{h_R}$ and $\nabla^{f_\eta^*}$ becomes the flat connection $d$ on $\R F_\phi$ as defined in section 3. Since $h_R$ restricts to the $\nu_0$ on $\textrm{End}_{h_R}(E)$, $f_R^*\nu$ becomes $h_R$. As well, we highlight that $m$ is defined to be a restriction of $\nu_0$.
We deduce that on $S-B$, $$ |\textrm{Re}_R \circ \nabla^{f_\eta^*}-\nabla^{f_R^*}  \circ \textrm{Re}_R|_{m^*\otimes f_R^*\nu,\sigma}=|\textrm{Re}_{h_R}\circ d-\nabla_{h_R}\circ \textrm{Re}_{h_R}|_{\nu_0^*\otimes h_R,\sigma}.$$
    Working with the right hand side of the equation above, since the $\pi_i$'s form an orthonormal basis for $\R F_\phi,$ Proposition \ref{dels} directly implies the claimed estimate.
\end{proof}

\begin{proof}[Proof of Theorem E]
For the first claim,
$f_R^*\nu-R^2f_\eta^*m= |R\phi|_{h_R}^2 - R^2|\eta|^2.$ Applying Proposition \ref{normcomparison} yields the stated comparison of $f_R^*\nu$ and $f_\eta^*m$. For the comparison between pullback connections, we apply Proposition \ref{prop:thmE2ndpart}. 
\end{proof}

\subsection{$\omega$-convergence of harmonic maps}
Toward Corollary E, we introduce asymptotic cones. For more exposition on the subject, see \cite{KL} and the references therein. Let $P(\mathbb{N})$ be the power set of $\mathbb{N}$. 
\begin{defn}
    A non-principal ultrafilter $\omega$ is a finitely additive probability measure on $P(\mathbb{N})$ such that $\omega(A)\in \{0,1\}$ for every $A\in P(\N)$ and $\omega(A)=0$ for every finite subset $A.$
\end{defn}
A choice of non-principal ultrafilter provides a way to assign limits to bounded sequences in $\R$, without having to pass to subsequences. Formally, a sequence $(x_n)_{n\in \N}$ in $\R$ has $\omega$-limit $x$, denoted $\lim_\omega x_n$, if for all neighbourhoods $U$ of $x,$ $\omega(\{n\in \N: x_n\in U\})=1.$ Every bounded sequence has a unique $\omega$-limit.

Let $(X_n,d_n,p_n)$, $n\in \mathbb{N}$, be a sequence of pointed metric spaces. A sequence of points $x_n\in X_n$, $n\in\mathbb{N}$, is called bounded if the sequence of non-negative numbers $d_n(x_n,p_n)$, $n\in \mathbb{N}$, is bounded.
\begin{defn}
    The asymptotic cone of $(X_n,d_n,p_n)$ with respect to a non-principal ultrafilter $\omega$ is the metric space $(X_\omega, d_\omega),$ where
    \begin{enumerate}
        \item points in $X_\omega$ are equivalence classes of sequences $(x_n)_{n=1}^\infty$ with $d_n(p_n,x_n)$ bounded, where two sequences $(x_n)_{n=1}^\infty$ and $(y_n)_{n=1}^\infty$ are equivalent if $\lim_\omega d_n(x_n,y_n)=0.$
        \item the distance $d_\omega$ between points $[x_n]$ and $[y_n]$ is $$d_\omega([x_n],[y_n])=\lim_\omega d_n(x_n,y_n).$$
    \end{enumerate}
\end{defn}
We often refer to $(X_\omega,d_\omega)$ as the limit of the $(X_n,d_n,p_n)'s$ with respect to $\omega$. It is well known that if the $(X_n,d_n)$'s are complete and NPC length spaces, then the same holds for $(X_\omega,d_\omega)$. Related to Corollary E, the main result we need about asymptotic cones is due to Kleiner and Leeb.
\begin{thm}[Theorem 5.1.1 in \cite{KL}]\label{thm: KL}
Let $(X,d)$ be the metric space associated to a symmetric space of non-compact type. Choosing a non-principal ultrafilter $\omega$, any sequence of basepoints $(p_n)_{n\in\mathbb{N}}\subset X$, and any sequence of scaling factors $(a_n)_{n\in \mathbb{N}}\subset (0,\infty)$, the asymptotic cone of $(X,a_n^{-1}d,p_n)$ with respect to $\omega$ is a Euclidean building modeled on a maximal split toral subalgebra of $\g$.
\end{thm}
\begin{remark}
    Assuming the continuum hypothesis, the asymptotic cone from Theorem \ref{thm: KL} is, up to isometries, independent of the choice of ultrafilter and the basepoint. Without the continuum hypothesis, there are uncountably many non-isometric asymptotic cones \cite{Tho}. 
\end{remark}
Given a sequence $(X_n,d_n,p_n)$ and ultrafilter $\omega$ that together give limit $(X_\omega,d_\omega),$ maps to $(X_\omega,d_\omega)$ can be constructed using sequences of maps to the $(X_n,d_n,p_n)$'s. Namely, given a metric space $(Z,d)$ and a family of locally uniformly Lipschitz maps $f_n: (Z,d)\to (X_n,d_n,p_n)$, $n\in \mathbb{N},$ so long as $$\sup_{n\in \mathbb{N}}d_n(f_n(x),p_n)<\infty$$ for some (and hence every) $x\in Z,$ the assignment $x\mapsto [f_n(x)]$ defines a locally Lipschitz map from $(Z,d)$ to $(X_\omega,d_\omega).$ Similarly, we can take limits of controlled group actions. Recall that if $\rho:\Gamma\to \textrm{Isom}(X,d)$ is an action of a group on a metric space, then the translation length of an element $\gamma\in \Gamma$ is $$\ell(\rho(\gamma))=\inf_{x\in X}d(\rho(\gamma)x,x).$$
Any sequence of actions $\rho_n$ of a group $\Gamma$ by isometries on $(X_n,d_n,p_n)$ with bounded translation lengths determines an isometric action $\rho_\omega$ on $(X_\omega,d_\omega)$ by $\rho_\omega ([x_n]) = [\rho_n(x_n)]$. 

Having discussed limits of metric spaces to asymptotic cones, we now explain how harmonic maps come along for the ride, the main point being Proposition \ref{prop: hmapsconvergence} below, which will be used in the next subsection in the context of symmetric spaces and buildings. 
\begin{prop}\label{prop: hmapsconvergence}
    Let $\omega$ be an ultrafilter and let $(X_n,d_n,p_n)$, $n\in\mathbb{N}$ be a sequence of pointed complete NPC length spaces with $\omega$-limit $(X_\omega,d_\omega)$. For each $n\in\mathbb{N},$ let $f_n:\tilde{S}\to (X_n,d_n)$ be a harmonic map. Assume that the $f_n's$ are locally uniformly Lipschitz and that for some $x\in \tilde{S},$ $\sup_{n\in\mathbb{N}}d_n(f_n(x),p_n)<\infty$. Then the map $f_\omega=[f_n]:\tilde{S}\to (X_\omega,d_\omega)$ is harmonic. 
\end{prop}

One can prove Proposition \ref{prop: hmapsconvergence} by adapting the mollification argument of Korevaar-Schoen from \cite[Theorem 3.9]{KS2}, replacing the Korevaar-Schoen limit space (see \cite[section 3]{KS2}) with the $\omega$-limit. Alternatively, as in \cite[section 3]{Kim}, one could deduce the proposition from the theorem of Korevaar-Schoen by embedding the Korevaar-Schoen limit space into the $\omega$-limit. We give a new proof using Lipschitz extension instead of mollification.

Proposition \ref{prop: hmapsconvergence} is straightforward once we have Lemma \ref{lem: sequence} below. We say that a pair of metric spaces $(Z,X)$ has the Lipschitz extension property if there exists $L\geq 1$ such that for every $k>0$ and $A\subset Z,$ every $k$-Lipschitz map $\varphi: A\to X$ admits an $Lk$-Lipschitz extension to a map from $Z\to X.$ For any NPC space $M,$ Lang-Schroeder's extension of the Kirszbraun theorem \cite[Theorem A]{zora22244} shows that the pair $(\mathbb{D},M)$ has the Lipschitz extension property with $L=1.$
\begin{lem}\label{lem: sequence}
    Let $Z$ be a separable metric space and let $\omega$ be an ultrafilter. Let $(X_n,d_n,p_n)$ be a sequence of pointed metric spaces with $\omega$-limit $X_\omega$ and $f_n: Z\to (X_n,d_n,p_n)$ a bounded and uniformly $A$-Lipschitz sequence with $\omega$-limit $f_\omega$. Assume that for each $n$, $(Z,X_n)$ has the Lipschitz extension property with a fixed constant $L$.
    
    Let $U\subset Z$ be an open and relatively compact subset and suppose that we are given a $B$-Lipschitz map $u:\overline{U}\to X_\omega$ that agrees with $f_\omega$ on $\partial U$. Setting $C(L,A,B)=3L\max\{A,2B\}$, there exists a bounded sequence of $C(L,A,B)$-Lipschitz maps $u_n: Z \to X_n$ such that $u=[u_n]$ and $u_n|_{\partial U}=f_n|_{\partial U}.$
\end{lem}
In \cite[Proposition 2.4]{LWY}, the authors show how to represent H{\"o}lder and Lipschitz maps to asymptotic cones by sequences of H{\"o}lder and Lipschitz maps. We adapt the procedure to add the boundary condition.
\begin{proof}
  Let $\{z_k:k\in \mathbb{N}\}\subset U$ be a countable and dense subset. For each $k,$ choose a point $z_k^*$ on $\partial U$ that minimizes the distance from $z_k$ to $\partial U$. Define, for $j\geq 2$, $$\delta_j = \min \{d(z_k,z_m), d(z_k^*,z_m^*):1\leq k < m \leq j\}.$$ For each $k\in \mathbb{N}$, choose a bounded sequence of points $x_{k,n}\in X_n$ such that $u(z_k)=[x_{k,n}]$. Inductively build a sequence of subsets $\mathbb{N}=N_1\supset N_2\supset\dots $ such that for each $j\geq 2$ we have $\omega(N_j)=1$ and for all $1\leq k,m\leq j$ and $n\in N_j$, $$|d_\omega(u(z_k),u(z_m))-d_n(x_{k,n},x_{m,n})|\leq \epsilon B\delta_j$$ and $$|d_\omega(u(z_k^*),u(z_m))-d_n(f_n(z_k^*),x_{m,n})|\leq \epsilon B\delta_j.$$ For every $j\in\mathbb{N}$, define $M_j=N_j\backslash N_{j+1},$ and $M_\infty=\cap_{i\in\mathbb{N}}N_i.$ We define a function $j:\mathbb{N}\to \mathbb{N}$ on $M_\infty$ by $j(n)=n$, and off $M_\infty$ by setting $j(n)$ to be the unique number such that $n\in M_{j(n)}.$ We define $u_n:\{z_1,\dots, z_{j(n)}\}\cup\partial U\to X_n$ by $u_n(z_k)=x_{k,n}$ and $u_n|_{\partial U}=f_n|_{\partial U}$.
  
  We prove that every $u_n$ is $3\max\{A,(1+\epsilon)B\}$-Lipschitz. Firstly, note that for all $1\leq k <m\leq j(n),$ $$d_n(u_n(z_k),u_n(z_m))\leq d_\omega(u(z_k),u(z_m))+\epsilon B \delta_{j(n)}\leq (1+\epsilon)Bd(z_k,z_m),$$ and the analogous inequality holds if we replace $z_m$ with $z_m^*.$ We thus see that $u_n$ is $(1+\epsilon)B$-Lipschitz on $\{z_1,\dots, z_{j(n)}\},$ and $A$-Lipschitz on $\partial U$. We're left to examine the Lipschitz constant we get when one point is a $z_i$ and the other is in $\partial U$. Given $z_i$ and $x\in \partial U$, 
  \begin{align*}
      d(u_n(z_i),u_n(x))&\leq d(u_n(z_i),u_n(z_i^*))+d(u_n(z_i^*),u_n(x))\leq (1+\epsilon)Bd(z_i,z_i^*) + Ad(z_i^*,x) \\
      &\leq \max\{A,(1+\epsilon)B\} (d(z_i,z_i^*) + d(z_i^*,x)).
  \end{align*}
Observing that $$d(z_i,z_i^*)+d(z_i^*,x)\leq 2d(z_i,z_i^*) + d(z_i,x) \leq 3d(z_i,x),$$ we get the claim.

  By the Lipschitz extension property, each $u_n$ extends to a $3L\max\{A,(1+\epsilon)B\}$-Lipschitz map from $\overline{U}\to X_n.$ Each $u_n$ takes $z_1$ to $x_{1,n},$ which together with Lipschitzness implies the distance to $p_n$ is controlled, and hence we can form the limit $u_\omega =[u_n].$ We obtain the lemma by setting $\epsilon=1.$
  \end{proof}

\begin{proof}[Proof of Proposition \ref{prop: hmapsconvergence}]
   First note that, by Korevaar-Schoen's construction of the locally $L^1$ metrics associated with Lipschitz maps, we have that for any relatively compact sub-surface $U\subset \tilde{S}$, $$\mathcal{E}_\omega (U, f_\omega)=\lim_\omega \mathcal{E}_n(U,f_n).$$ Above, we use $\mathcal{E}_n$ and $\mathcal{E}_\omega$ for energies taken with respect to the targets $(X_n,d_n)$ and $(X_\omega, d_\omega)$ respectively. Now, for $U\subset \tilde{S}$ relatively compact, let $u:\overline{U}\to (X_\omega,d_\omega)$ be a Lipschitz map that agrees with $f_\omega$ on $\partial U,$ i.e., a competitor for the energy. By Lemma \ref{lem: sequence}, one can represent $u$ by a sequence of Lipschitz maps $u=[u_n],$ where $u_n|_{\partial U}=f_n|_{\partial U}$. To compare the energy of $f_\omega$ to that of $u$, we have that, since each $f_n$ is harmonic, $$\mathcal{E}_n(U,f_n)\leq \mathcal{E}_n(U,u_n).$$ Taking the limit with respect to $\omega$ yields that $$\mathcal{E}_\omega (U, f_\omega)=\lim_\omega \mathcal{E}_n(U,f_n)\leq \lim_\omega \mathcal{E}_n(U,u_n)=\mathcal{E}_\omega(U,u),$$ which establishes that $f_\omega$ is harmonic.
\end{proof}
Note that, in Proposition \ref{prop: hmapsconvergence}, if each $f_n$ is equivariant for an action $\rho_n: \pi_1(S)\to \textrm{Isom}(X_n,d_n),$ then each $\rho_n$ has bounded translation lengths and moreover $f_\omega$ is equivariant for the limit action $\rho_\omega=[\rho_n].$

\subsection{Proof of Corollary E}
We prove Corollary E, restated below. We return to the set-up from sections 8.1 and 8.2, and we now choose a sequence $(R_i)_{i=1}^\infty\subset \R^+$ tending to infinity, and assume that the harmonic maps $f_i=f_{R_i}:\tilde{S}\to (N,\nu)$ are equivariant with respect to representations $\rho_i:\pi_1(S)\to G$. Choose a a non-principal ultrafilter $\omega$ and scaling factors $a_i=R_i^{-1}.$ Fixing $x\in \tilde{S}$, in order to apply Proposition \ref{prop: hmapsconvergence} to $f_i:\tilde{S}\to (N,R_i^{-1}\nu, f_i(p)),$ we need some rough control on the Lipschitz constants of the $f_i$'s that holds everywhere and not just away from the critical set.
\begin{prop}\label{roughlipbound}
    On a fixed relatively compact open subset $\Omega\subset S$, there exists a constant $C=C(G,\Omega,\sigma)>0$ such that the maps $f_i$ are uniformly $CR_i$-Lipschitz. 
\end{prop}
\begin{proof}
Let $U$ be any relatively compact conformal disk intersecting $\Omega$. By Proposition \ref{firstbound}, there exists $C_j=C_j(G,U,\sigma)>0,$ $j=1,2,$ such that, on $U$,
 $$|R_i\phi|_{h_{R_i},\sigma}^2\leq C_1R_i^2\max_{z\in U}|\eta(z)|_{\sigma}^2 +C_2.$$ Hence, on $U$, $$\textrm{Lip}(f_i)^2\leq \max_{z\in U} e(f_i)(z)=\max_{z\in U} R_i^2|\phi(z)|_{h_{R_i},\sigma}^2\leq C_1R_i^2\max_{z\in U}|\eta(z)|_{\sigma}^2 +C_2.$$ By covering the closure of $\Omega$ with finitely many disks, we obtain a uniform bound on $\textrm{Lip}(f_i)$ over all of $\Omega$.
\end{proof}
 Applying Theorem \ref{thm: KL} and Proposition \ref{prop: hmapsconvergence} to the sequence above, we obtain a building $(\mathcal{B},d_\omega)$ as the asymptotic cone of our sequence of rescaled pointed symmetric spaces, together with the limiting action $\rho_\omega=[\rho_i]:\pi_1(S)\to \textrm{Isom}(\mathcal{B},d_\omega)$ and the $\rho_\omega$-equivariant harmonic map $f_\omega=[f_{R_i}]: \tilde{S}\to (\mathcal{B},d_\omega).$ As in the introduction, we set $g_{f_\omega}=g(f_\omega)$ to be the $L^1$ measurable pullback metric.
 \begin{cor}[Corollary E]
     As measurable tensors, $g_{f_\omega}=f_\eta^*m,$ and the $\mathfrak{t}^{\C}//W$-valued $1$-form of $f_\omega$ is equal to $\eta$.
 \end{cor}

\begin{proof}[Proof of Corollary E]
By the construction of $L^1$ measurable pullback metrics in Korevaar-Schoen's work \cite{KS}, for any point $p\in S$ and any tangent vector $v\in T_pS,$ $$g_{f_\omega}(v,v)=\lim_{\omega}R_i^{-2} f_i^*\nu(v,v).$$
    The first part of Theorem E thus implies that $g_{f_{\omega}}=f_\eta^*m$. For the $\ft^{\C}//W$-valued $1$-forms, for every $i$ let $h_{R_i}$ be the harmonic metric obtained by taking our harmonic $G$-bundle to an ordinary harmonic bundle via the adjoint representation. Each harmonic metric determines a harmonic map to the symmetric space of $\textrm{SL}(\g).$ Using the same principal ultrafilter $\omega$ and going through Proposition \ref{roughlipbound} again, we obtain a new harmonic map to a building $f_\omega':\tilde{S}\to (\mathcal{B}',d'),$ where $(\mathcal{B}',d')$ is the asymptotic cone of the symmetric space of $\textrm{SL}(\g)$. 
    Under the adjoint representation, $\ft$ is injectively mapped into a maximal split toral subalgebra $\mathfrak{a}$ of $\mathfrak{sl}(\g)$. The $\ft^{\C}//W$-valued $1$-form of $f_\omega$ is recovered from that of $f_\omega'$, which is a $\mathfrak{a}^{\C}//S_n$-valued $1$-form, by undoing (the complexification of) this mapping.
    
    By Proposition \ref{prop: singular set}, the regular set $S_{reg}$ of $f_\omega'$ is non-empty. Fixing a small contractible disk $U$ in $S_{reg}-B$ and a lift $V$ in $\tilde{S}$, let $x$ and $y$ be points in $U$ with lifts $\tilde{x}$ and $\tilde{y}$ in $\tilde{U}$, and let $\Pi_{R_i}$ be the parallel transport from $x$ to $y$ coming from the flat connection of our harmonic bundle. The vector distance between the harmonic metrics $h_{R_i}(x)$ and $\Pi_{R_i}^*h_{R_i}(y)$, rescaled by $R_i^{-1},$ limits precisely as $R_i\to \infty$ to the Weyl-chamber valued distance between points $f_\omega'(\tilde{x})$ and $f_\omega'(\tilde{y})$ in the limiting building.
 Thus, from Theorem D, the solution to the WKB problem, if $\phi_1,\dots, \phi_{\textrm{dim}\g}$ are the eigen-$1$-forms (with multiplicity) of the adjoint representation of $\phi,$ then $f_\omega'$ is locally described as, up to translation and reflections, $$f_\omega'(z) = (\int_{z_0}^z\textrm{Re}\phi_1,\dots, \int_{z_0}^z \textrm{Re}\phi_{\textrm{dim}\g} ).$$  Taking $\partial$ in the affine chart then yields the $\mathfrak{a}^{\C}//S_n$-valued $1$-form associated with $f_\omega'|_U.$ Undoing the mapping the $\ft^{\C}\to \mathfrak{a}^{\C},$ we see that the $\ft^{\C}//W$-valued $1$-form of $f_\omega|_U$ is $\eta$. By holomorphicity, the $\ft^{\C}//W$-valued $1$-form agrees with $\eta$ everywhere.
\end{proof}
\begin{remark}
 According to \cite[Proposition 3.23]{KNPS}, for $G=\textrm{SL}(n,\C)$, the singular set of a harmonic map to a building with $\ft^{\C}//W$-valued $1$-form $\eta$ is always contained in the critical set.
\end{remark}

\appendix

\section{Hitchin WKB Problem}

We prove Theorem D. Using the new approximation given by Corollary C and the bounds given by Theorem B, the proof becomes a straightforward adaption of Mochizuki's work in \cite{Mo}. The key input from \cite{Mo} is \cite[Corollary 2.19]{Mo}.

We resume notation and conventions from section \ref{HWKBsection}: $(E,\overline{\partial}_E,R\phi,h_R),$ $R>0$ is a ray of harmonic bundles of rank $n$ over a Riemann surface $S,$ and $\gamma:[0,1]\to S$ is a path in the complement of the critical set. We set $\phi_1,\dots, \phi_n$ to be the eigen-$1$-forms of $\phi$ defined in an open neighbourhood of $\gamma$, and we assume that $\gamma$ is non-critical, which means that $\mathrm{Re}(\phi_i(\dot{\gamma}))$ is never equal to $\mathrm{Re}(\phi_j(\dot{\gamma}))$ for $\phi_i\neq \phi_j$. As in section \ref{HWKBsection}, we set $\alpha_i=-\int_\gamma \textrm{Re}(\phi_i),$ ordered so that $\alpha_1\geq \alpha_2\geq \dots \geq \alpha_n.$ For each $R>0,$  $\Pi_{R,\gamma}: E_{\gamma(0)}\to E_{\gamma(1)}$ is the parallel transport operator induced by the flat connection $D_{h_R}$.

We're looking to compute the vector distance, i.e., the Weyl-chamber valued distance, between $(h_R)_{\gamma(0)}$ an $\Pi_{R,\gamma}^*(h_R)_{\gamma(1)}$. In \cite[section 1]{KNPS}, Katzarkov, Noll, Pandit, and Simpson outline how to compute such distances. Given a linear map $f$ between Hermitian vector spaces $(V_i,h_i)$, $i=1,2$, let $||f||_{op}$ be the operator norm. In the work below the metrics will be implicitly understood, so we're okay suppressing the metrics from the notation $||\cdot||_{op}.$ Write $\vec{d}((h_R)_{\gamma(0)},\Pi_{R,\gamma}^*(h_R)_{\gamma(1)})=(\beta_1^R,\dots, \beta_n^R).$ Observe that 
\begin{equation}\label{beta1formula}
    \beta_1^R = \log ||\Pi_{R,\gamma}||_{op}.
\end{equation}
Moreover, for any $k\leq l,$ $\Pi_\gamma$ induces a homomorphism $\wedge^k\Pi_\gamma$ of $\wedge^k E_{\gamma(0)},$ and 
\begin{equation}\label{betakformula}
    \sum_{i=1}^k\beta_i^R = \log ||\wedge^k\Pi_{R,\gamma}||_{op}.
\end{equation}
Hence, for every $k,$ 
\begin{equation}\label{differenceeq}
    \beta_k^R = ||\wedge^k\Pi_{R,\gamma}||_{op}- ||\wedge^{k-1}\Pi_{R,\gamma}||_{op}.
\end{equation}
Thus, to solve the Hitchin WKB problem, we need to estimate $||\wedge^k\Pi_{R,\gamma}||_{op}$ for every $k.$ 

\subsection{Preliminary results}
Following \cite[section 2.4]{Mo}, we define
\begin{itemize}
    \item $M(n,\C)$ to be the set of $n\times n$ matrices with entries in $\C$,
    \item $M(n,\C)_0=\{(a_{ij})_{i,j=1}^n\in M(n,C): a_{ij}=0 \textrm{ for }i\neq j\}$ to be the set of diagonal matrices, and 
    \item $M(n,\C)_1=\{(a_{ij})_{i,j=1}^n\in M(n,C): a_{ij}=0 \textrm{ for }i=j\}$ to be the set of matrices such that all diagonal entries are zero.
\end{itemize}
We fix any norm on the finite dimensional vector space $M(n,\C),$ which we then use to define $C^{k}$-norms $||\cdot ||_k$ for maps from $[0,1]\to M(n,\C)$. 

Next, let $C\geq 0$ be a constant and consider functions $a_j,b_j\in C^0([0,1])$, $j=1,\dots, n,$ such that 
\begin{enumerate}
    \item $\textrm{Re}(a_1(s))\leq \textrm{Re}(a_2(s))\leq\dots \leq \textrm{Re}(a_n(s))$ for any $s$, and
    \item $|b_j(s)|\leq C.$
\end{enumerate}
For $t\geq 0$, put $\alpha_j^t(s):=ta_j(s)+b_j(s)$. Let $A^t$ be the $n\times n$ matrix with diagonal entries $\alpha_j^t.$ 

Let $V$ be a vector bundle with connection $\nabla$. Recall that in a frame $w=(w_1,\dots,w_n)$, the connection is expressed $\nabla=d+A,$ where $A=(A_{ij})$ is a matrix valued $1$-form, and we have the defining relation $\nabla w_i=\sum_jA_{ij}w_j.$
\begin{lem}[Corollary 2.19 in \cite{Mo}]\label{cor2.19}
    There exists $C_1,\epsilon_0>0$ depending only on $C$ such that the following holds. Let $E$ be a $C^1$ vector bundle on $[0,1]$ with frame $v=(v_1,\dots, v_n).$ Let $B:[0,1]\to M(n,\C)_1$ be a continuous map such that $||B||_0\leq \epsilon_0$. For $t\geq 0,$ let $\nabla^t$ be the connection defined in the frame $v$ by $$\nabla^t = d+(A^t(s)+B(s))ds.$$ There exists $G^t\in C^1([0,1],M(n,\C)_1)$ and $H^t\in C^0([0,1],M(n,\C)_0)$ such that
    \begin{enumerate}
        \item $||G^t||_0+||\partial_s G^t+[A^t,G^t]||_0+||H^t||_0\leq C_1||B||_0$, and
        \item in the frame $u^t(s)=(I+G^t(s))^{-1}v(s),$ the connection form of $\nabla^t$ is $(A^t(s)+H^t(s))ds$
    \end{enumerate}
\end{lem}
\begin{remark}
    Mochizuki assumes $\textrm{Re}(a_1)<\textrm{Re}(a_2)<\dots < \textrm{Re}(a_n).$ The only place he uses this assumption is in the proof of his Lemma 2.21, and only $\textrm{Re}(a_1)\leq\textrm{Re}(a_2)\leq\dots \leq \textrm{Re}(a_n)$ is needed.
\end{remark}

\subsection{Solution}
\begin{proof}[Proof of Theorem D]
Choose a precompact and contractible open neighbourhood $U\subset S$ that does not intersect the critical set and that contains the image of our non-critical path $\gamma([0,1])$. We equip $U$ with a flat metric, which we use to measure norms. Extract open balls $B_1(r),\dots, B_m(r)$ of small radius $r>0$ such that 
\begin{enumerate}
    \item $\cup_i B_i(r)$ does not intersect the critical set, and 
    \item $U\subset \cup_i B_i(r/2)$.
\end{enumerate}
Since $\cup_i B_i(r/2)$ has a definite distance from the critical set, we can find positive numbers $A$ and $b$ such that in restriction to each $B_j(r/2)$, $R\phi$ lies in $S(bR,A)$. As we chose $U$ to be contractible, over $U$ we can split $E$ into generalized eigenspaces as $E = \oplus_{i=1}^m E_i$, with eigen-$1$-forms $\phi_i$ and projections $\pi_i$. For each $i$, let $h_R^i$ be the restriction of $h_R$ to $E_i$ and $h_R^{\oplus}=\oplus_{i=1}^m h_R^i$ and let $\nabla_{h_R}^{\oplus}$ be the  Chern connection. Write $*_R^{\oplus}$ for the adjoint operator of $h_R^{\oplus}.$ Corollary C implies that on $U,$ $D_{h_R}$ differs from
\begin{equation}\label{connectionestimate}
    \nabla_{h_R}^{\oplus} +\sum_{i=1}^m R(\phi_i+\overline{\phi}_i)\pi_i + R\phi_n+R\phi_n^{*_R^{\oplus}}
\end{equation}
by an $\textrm{End}(E|_U)$ valued $1$-form whose $h_R$-norm is on the order of $Ce^{-cR},$ where $C$ and $c$ depend on $n,r$, and the eigen-$1$-forms of $\phi$. Note that Theorem B implies that $|R\phi_n|_{h_R}$ is uniformly controlled, and the proof of Corollary C demonstrates $|R\phi_n^{*_R^{\oplus}}|_{h_R}$ differs from $|R\phi_n^{*_{h_R}}|_{h_R}=|R\phi_n|_{h_R}$ by a term that decays like $Ce^{-cR}.$

To compute the operator norms of parallel transport operators, we will work on the pullback vector bundle $\gamma^*E=\oplus_{i=1}^m\gamma^*E_i$ over $[0,1]$ with connections $\gamma^*D_{h_R}$ and $\gamma^*\nabla_{h_R},$ and choose good frames. Specifically, we take $h_i$-orthonormal frames $u_i=(u_{i1},\dots, u_{im_i})$ for $E_i$ that are $\gamma^*\nabla_i$ parallel, which means that $\gamma^*\nabla_i u_{ij}=0$ for all $i,j$. We then pull these back to frames for $\gamma^*E_i$, and combine all of the frames for the $\gamma^*E_i$'s to get a frame for $\gamma^*E$. We then normalize the frame dividing every $u_{ij}$ by $R$. Setting $D_{h_R}^{\oplus}=  \nabla_{h_R}^{\oplus} +\sum_{i=1}^m R(\phi_i+\overline{\phi}_i)\pi_i$, with respect to this frame, $\gamma^* D_{h_R}^\oplus$ has diagonal connection form $$A(s)ds = \textrm{diag}(2\textrm{Re}\phi_1(s),\dots, 2\textrm{Re}\phi_n(s)).$$ By the comments around (\ref{connectionestimate}), in this frame, the connection $\gamma^*D_{h_R}$ has connection form $$(A(s)+B_0(s)+B_1(s))ds,$$ where 
\begin{enumerate}
\item $B_0$ is a $C^0$ map from $[0,1]\to M(n,\C)_0$ that is $C^\infty$ on $(0,1)$, and
\item $B_1$ is a $C^0$ map from $[0,1]\to M(n,\C)_1$ that is $C^\infty$ on $(0,1)$. 
\end{enumerate}
Using our comparison to (\ref{connectionestimate}), as well as our uniform controls, there exists $C>0$ with the same dependencies as above such that both $B_0$ and $B_1$ satisfy $$|B_m(s)|\leq CR^{-1}.$$ Taking $R$ sufficiently large, we may assume that both $||B_m||_0$  are small enough to satisfy the hypothesis of Lemma \ref{cor2.19}. Invoking Lemma \ref{cor2.19}, we know that there exists a $C^1$ function $G:[0,1]\to M(r,\C)$, a $C^0$ function $H:[0,1]\to M(n,\C)_0$, and a constant $C_1$ such that 
\begin{enumerate}
\item in $C^1$-norm, $||G||_{C^1}<C_1R^{-1}$,
\item in $C^0$ norm, $||H||_{C^0}<C_1R^{-1}$, and
\item with respect to the frame $v=(I+G)u$, the connection form of $\gamma^*D_{h_R}$ is $$(A(s)+B_0(s)+H(s))ds.$$
\end{enumerate}
Next, note that the frame $$v'=\textrm{exp}\Big (-\int_0^s(A(p)+B_0(p)+H(p))dp\Big )v$$ is flat for $\gamma^*D_{h_R}$. Therefore, the matrix of the parallel transport operator for $\gamma^*D_{h_R}$ from $\gamma(0)$ to $\gamma(1)$ is $$(I+G(1))\textrm{exp}\Big (-\int_0^s(A+B_0+H)dp\Big )(I+G(0))^{-1}.$$ Using the Gram-Schmidt process, we transform the frame $v'$ into an orthogonal frame such that all vectors have norm $R^{-1}.$ Denoting the new frame by $q$, it can be described by $$q(s)=(1+K(s))v'(s),$$ where $||K(s)||_{C^0}$ decays like $R^{-1}.$ If we set $$L(s)=(1+K(s))(I+G(s))-I,$$ then $\Pi_{R,\gamma}$ is represented in the frame $q$ by $$Z_{\gamma}:=(1+L(1))\textrm{exp}\Big (-\int_0^s(A(p)+B_0(p)+H(p))dp\Big )(I+L(0))^{-1}.$$ With this representation, we directly compute operator norms. For some constant $C>0,$ since $||L(s)||_{C^0}$ is of order at most $R^{-1}$, we have
$$\left|\log ||Z_\gamma||_{op} - \log ||\textrm{exp}\Big (-\int_0^s(A(p)+B_0(p)+H(p))dp\Big )||_{op}\right| \leq CR^{-1}.$$
Next, using the bounds on $B_0$ and $H$, together with the definition of $\alpha_1$ and formula (\ref{beta1formula}), 
$$\left|\log ||\textrm{exp}\Big (-\int_0^s(A(p)+B_0(p)+H(p))dp\Big )||_{op} - 2\alpha_1\right| \leq CR^{-1}.$$ 
Putting these together,
$$|\log ||Z_\gamma||_{op}- 2\alpha_1|\leq CR^{-1}.$$
Estimating $||\wedge^k Z_\gamma||_{op}$ in the same fashion, but with (\ref{betakformula}), we find that for all $k$, $$|\log ||\wedge^k \Pi_{R,\gamma}||_{op}
-2\sum_{j=1}^k\alpha_j|\leq CR^{-1}.$$ Theorem D then follows from inserting the inequalities above into the formula (\ref{differenceeq}). To understand the dependencies, note that our choices for $r>0$ depend on $S$ and $\gamma$.
\end{proof}
\begin{remark}
In the proof above, we can decompose the connection form of $\gamma^*D_{h_R}$ further: $B_1$ decomposes according to the decomposition $\textrm{End}(E)=\oplus_{i,j=1}^m\textrm{Hom}(E_i,E_j)$. It should be possible to prove a refinement of Lemma \ref{cor2.19} in which we give $G$ and $H$ the same decay as the connection form in any $\textrm{Hom}(E_i,E_j)$ block.

Such a refinement would lead to a more precise version of Theorem D: on the entries of $\frac{1}{R}\vec{d}((h_R)_{\gamma(0)},\Pi_{R,\gamma}^*(h_R)_{\gamma(1)})$, corresponding to a subspace $E_i$ such that $\phi|_{E_i}$ is semisimple, the convergence should be exponential in $R$.
\end{remark}

\bibliographystyle{plain}
\bibliography{bibliography}

\end{document}